\documentclass
{amsart}

\usepackage{amssymb, enumerate, xspace, graphicx, url, microtype, stmaryrd}
\usepackage[T1]{fontenc}
\usepackage[latin1]{inputenc}
\usepackage[all]{xy}
\SelectTips{cm}{}

\usepackage[active]{srcltx}
\numberwithin{equation}{section}

\numberwithin{subsection}{section}

\allowdisplaybreaks[1]

\newenvironment{enumeratea}
{\begin{enumerate}[\upshape (a)]}
{\end{enumerate}}

\newenvironment{enumerate1}
{\begin{enumerate}[\upshape (1)]}
{\end{enumerate}}

\newcommand{\theoremname}{testing}
\newenvironment{named}[1]{\renewcommand\theoremname{#1}
\begin{namedtheorem}}
{\end{namedtheorem}}
\newtheorem*{namedtheorem}{\theoremname}

\newtheorem{theorem}{Theorem}[section]
\newtheorem{proposition}[theorem]{Proposition}
\newtheorem{proposition-definition}[theorem]
{Proposition-Definition}
\newtheorem{corollary}[theorem]{Corollary}
\newtheorem{lemma}[theorem]{Lemma}

\theoremstyle{definition}
\newtheorem{definition}[theorem]{Definition}

\newtheorem{examples}[theorem]{Examples}
\newtheorem{remark}[theorem]{Remark}

\newtheorem{conjecture}[theorem]{Conjecture}

\theoremstyle{remark}

\newcommand\nome{testing}
\newcommand\call[1]{\label{#1}\renewcommand\nome{#1}}
\newcommand\itemref[1]{\item\label{\nome;#1}}
\newcommand\refall[2]{\ref{#1}~(\ref{#1;#2})}
\newcommand\refpart[2]{(\ref{#1;#2})}

\newcommand\cA{\mathcal{A}} \newcommand\cB{\mathcal{B}}

\newcommand\cO{\mathcal{O}}

\newcommand\CC{\mathbb{C}} 
 
\newcommand\GG{\mathbb{G}}

 \newcommand\NN{\mathbb{N}}
 \newcommand\PP{\mathbb{P}}
\newcommand\QQ{\mathbb{Q}}

 \newcommand\ZZ{\mathbb{Z}}

 \newcommand\rD{\mathrm{D}}

\newcommand\rK{\mathrm{K}}

 \newcommand\rR{\mathrm{R}}
\newcommand\rS{\mathrm{S}}

\newcommand\rma{\mathrm{a}} \newcommand\rmb{\mathrm{b}}

\newcommand\rmm{\mathrm{m}}

\newcommand\arr{\ifinner\to\else\longrightarrow\fi}

\newcommand\arrto{\ifinner\mapsto\else\longmapsto\fi}

\newcommand\darr{\ifinner\Rightarrow\else\Longrightarrow\fi}

\newcommand{\xarr}{\xrightarrow}

\renewcommand\H{\operatorname{H}}

\newcommand\op{^{\mathrm{op}}}

\newcommand\eqdef{\overset{\mathrm{\scriptscriptstyle def}} =}

\renewcommand\th{^\text{th}}

\def\displaytimes_#1{\mathrel{\mathop{\times}\limits_{#1}}}

\def\displayotimes_#1{\mathrel{\mathop{\bigotimes}\limits_{#1}}}

\renewcommand\hom{\operatorname{Hom}}

\newcommand\End{\operatorname{End}}

\newcommand\aut{\operatorname{Aut}}

\newcommand\pic{\operatorname{Pic}}

\newcommand\spec{\operatorname{Spec}}

\newcommand\generate[1]{\langle #1 \rangle}

\newcommand\id{\mathrm{id}}

\newcommand\pr{\operatorname{pr}}

\newcommand\indlim{\varinjlim}

\renewcommand\projlim{\varprojlim}

\newcommand{\GL}{\mathrm{GL}}

\newcommand\dash{\nobreakdash-\hspace{0pt}}

\newdir{ >}{{}*!/-5pt/@{>}}

\newcommand\double{\mathbin{\rightrightarrows}}

\newcommand\doublelong[2]{\mathbin{\xymatrix{{}\ar@<3pt>[r]^{#1}
\ar@<-3pt>[r]_{#2}&}}}

\newcommand{\underhom}
{\mathop{\underline{\mathrm{Hom}}}\nolimits}

\newcommand{\underisom}
{\mathop{\underline{\mathrm{Isom}}}\nolimits}

\newcommand{\underaut}
{\mathop{\underline{\mathrm{Aut}}}\nolimits}

\newlength{\ignora}
\newcommand{\hsmash}[1]{\settowidth{\ignora}{#1}#1\hspace{-\ignora}}

\newcounter{stepcount}

\newcommand{\catsch}[1]{(\mathrm{Sch}/#1)}

\newcommand{\mmu}{\boldsymbol{\mu}}

\newcommand{\et}{{\rm \acute{e}t}}

\newcommand{\gm}{\GG_{\rmm}}

\newcommand{\cha}{\operatorname{char}}

\newcommand{\univ}[2][\kappa]{\Pi_{#2/#1}}
\newcommand{\etuniv}[2][\kappa]{\Pi_{#2/#1}^{\mathrm{\acute{e}t}}}
\newcommand{\tuniv}[2][\kappa]{\Pi_{#2/#1}^{\mathrm{tame}}}

\newcommand{\rep}{\operatorname{Rep}}

\newcommand{\vect}{\operatorname{Vect}}
\newcommand{\coh}{\operatorname{Coh}}

\newcommand{\infl}{inflexible\xspace}

\newcommand{\geom}{geometrically connected and geometrically reduced\xspace}

\newcommand{\ikt}{\widetilde{\rK}_{0}}

\newcommand{\fin}{\operatorname{Fin}}

\newcommand{\efin}{\operatorname{EFin}}

\newcommand{\tfin}{\operatorname{TFin}}

\newcommand{\sym}{\operatorname{Sym}}

\newcommand{\ga}{\GG_{\rma}}

\newcommand{\affger}{(\mathrm{Aff\,Ger}/\kappa)}

\newcommand{\psp}{pseudo-proper\xspace}

\newcommand{\aff}[1]{(\mathrm{Aff}/#1)}

\newcommand{\catset}{\mathrm{(Set)}}

\newcommand{\thickslash}{\mathbin{\!\!\fatslash}}

\newcommand{\fet}{_{\rm f\acute{e}t}}

\newcommand{\LC}{\operatorname{LC}}

\newcommand{\RH}{\operatorname{RH}}

\DeclareMathOperator{\homext}{{Hom-ext}}
\DeclareMathOperator{\HomExt}{\mathbf{Hom-ext}}
\DeclareMathOperator{\IsCl}{IsCl}
\DeclareMathOperator{\tors}{Tors}
\DeclareMathOperator{\h}{H}

\newcommand{\sep}{^{\mathrm{sep}}}

\begin{document}

\bibliographystyle{amsalpha}

\title[The Nori fundamental gerbe of a fibered category]{The Nori fundamental gerbe\\of a fibered category}

\author[Borne]{Niels Borne$^\dagger$}

\author[Vistoli]{Angelo Vistoli$\ddagger$}

\address[Borne]{Laboratoire Paul Painlevé\\
Université de Lille\\
U.M.R. CNRS 8524\\
U.F.R. de Mathématiques\\
59\,655 Villeneuve d'Ascq Cédex\\
France}
\email{Niels.Borne@math.univ-lille1.fr}

\address[Vistoli]{Scuola Normale Superiore\\Piazza dei Cavalieri 7\\
56126 Pisa\\ Italy}
\email{angelo.vistoli@sns.it}

\thanks{$^\dagger$Supported in part by the Labex CEMPI (ANR-11-LABX-0007-01) and Anr ARIVAF (ANR-10-JCJC 0107)}

\thanks{$^\ddagger$Supported in part by the PRIN project ``Geometria delle varietà algebriche e dei loro spazi di moduli'' from MIUR}

\maketitle

\begin{abstract}
We extend Nori's theory of the fundamental group scheme to a theory of the fundamental gerbe, which applies to schemes, algebraic stacks, and more general fibered categories, even in absence of a rational point. We give a tannakian interpretation of the fundamental gerbe in terms of essentially finite bundles, extending Nori's correspondence for complete varieties with a rational point. We also show how our formalism allows a natural formulation of Grothendieck's section conjecture in arbitrary characteristic.
\end{abstract}

\section{Introduction}

Let $X$ be a reduced proper connected scheme over a perfect field $\kappa$, with a rational point $x_{0} \in X(\kappa)$. The celebrated result of Nori \cite{nori-phd} says the following.

\begin{enumeratea}

	\item There is a profinite group scheme $\pi(X, x_{0})$, the \emph{Nori fundamental group scheme}, with a $\pi(X,x_{0})$-torsor $P \arr X$ with a trivialization $P\mid_{x_{0}} \simeq \pi(X, x_{0})$ such that for every finite group scheme $G \arr \spec \kappa$ and every $G$ torsor $Q \arr X$ with a trivialization $\alpha\colon Q \mid_{x_{0}} \simeq G$, there is a unique homomorphism of group schemes $\pi(X, x_{0}) \arr G$ inducing $Q$ and $\alpha$.

\item There is an equivalence of Tannaka categories between representations of the group scheme $\pi(X, x_{0})$ and essentially finite locally free sheaves on $X$.

\end{enumeratea} 

Let us recall that a locally free sheaf $E$ on $X$ is \emph{finite} if there exist two polynomials $f$ and $g$ in one variable with natural numbers as coefficients, with $f \neq g$, such that $f(E) \simeq g(E)$ (here we evaluate $f$ on $E$ by using direct sums and tensor powers). The notion of \emph{essentially finite} is more delicate, and we refer to \cite[p.~82]{nori-phd} for the definition, which uses a notion of semistable locally free sheaf. In characteristic~0, it turns out, as a consequence of the theorem, that every essentially finite sheaf is in fact finite.

In this paper we generalize Nori's construction by removing the assumption that $X$ has a rational point $x_{0} \in X(\kappa)$, and we give a simpler and more direct approach to Nori's correspondence between representations and essentially finite locally free sheaves. We also show how our formalism allows to give a natural interpretation of Grothendieck's Section Conjecture, and a natural formulation of the conjecture in arbitrary characteristic.

As to the first point, we replace the group with a gerbe (thus we use the language of gerbes, not that of groupoids). In characteristic~0 this gerbe is the one associated with the fundamental groupoid of Deligne \cite{deligne-minus-three-points}; an alternative construction in the smooth case is due to Esnault and Hai \cite{esnault-hai-groupoid}. The gerbes that appear are not of finite type, so we need to consider gerbes in the uncomfortably large fpqc topology. Fortunately, we don't need to use dubious notions such as the fpqc stackification. 

In our approach we don't need $X$ to be a scheme, we work with fibered categories over a fixed field $\kappa$ (not necessarily stacks, the sheaf conditions on $X$ don't play any role). After some preliminary work on fpqc gerbes and projective limits, and on finite stacks, in Section~\ref{sec:fund-gerbe} we introduce a class of fibered categories, called \emph{\infl} (Definition~\ref{def:fund-gerbe}). This notion has some geometric content: the class of \infl algebraic stacks is properly contained in the class of geometrically connected algebraic stacks, and properly contains the class of \geom algebraic stacks. Then we define a \emph{fundamental gerbe} of a fibered category $X$ as a profinite gerbe $\univ X$ with a morphism $X \arr \univ X$, such that every morphism from $X$ to a finite stack $\Gamma$ factors uniquely through $\univ X$. The fundamental gerbe is unique. Our main result in this section is that a fibered category has a fundamental gerbe if and only if it is \infl. In Section~\ref{sec:base-change} we show that the fundamental base gerbe has a very useful base-change property with respect to algebraic separable extensions.

Next we discuss the tannakian interpretation of the fundamental gerbe, generalizing Nori's approach. Let us recall that Grothendieck, Saavedra Rivano and Deligne (\cite{saavedra}, \cite{deligne-tannakian}) have shown that there is an equivalence between non-neutral Tannaka categories on one side and affine fpqc gerbes on the other; thus it is natural to ask whether the gerbe $\univ X$ has a tannakian interpretation. For this we need to work with finite locally free sheaves on $X$, and this only works if we impose a finiteness condition on $X$. 

In Section~\ref{sec:tannaka} we introduce a class of fibered categories, that we call \emph{pseudo-proper}: a fibered category over $\kappa$ is \psp when it has a fpqc cover by a quasi-compact and quasi-separated scheme, and furthermore for any locally free sheaf $E$ on $X$ the dimension of the $\kappa$-vector space $\H^0(X,E)$ is finite. This last condition is a very natural one to impose, as it ensures that the Krull--Schmidt theorem holds for locally free sheaves on $X$.

Then we discuss finite locally free sheaves on a \psp fibered category, copying Nori's definition. We also define essentially finite locally free sheaves, with a definition that is more general and simpler than Nori's, and does not use semistable locally free sheaves at all. (In the case of a complete scheme over $\kappa$ our notion and Nori's turn out to coincide.) We denote by $\efin X$ the category of essentially finite locally free sheaves.

The following is our main result.

\begin{named}{Main theorem}[Theorems \ref{thm:tannakian-description} and \ref{thm:char-infl}]
Let $X$ be a \psp fibered category.

\begin{enumeratea}

\item The fibered category $X$ is \infl if and only if $\efin X$ is tannakian. 

\item If $X$ is \infl, then $\efin X$ is equivalent as a Tannaka category to the category of representations of the fundamental gerbe $\univ X$.

\item If $\cha \kappa = 0$ and $X$ is \infl, then every essentially finite locally free on $X$ is in fact finite.

\end{enumeratea}
\end{named}

We find it interesting that the condition of being inflexible is in fact equivalent to the completely different condition that $\efin X$ be tannakian; this seems to suggest that indeed this inflexibility condition is a very natural one.

Next we study two natural quotients of $\univ X$. The first is the largest étale quotient $\etuniv X$ of $\univ X$; in Section~\ref{sec:etale} we show that this coincides with the stack associated with Deligne's relative fundamental groupoid, introduced in~\cite{deligne-minus-three-points}.

In Section~\ref{sec:section}, we show how Grothendieck's famous Section Conjecture can be interpreted as a statement about the étale fundamental gerbe (Conjecture~\ref{conj:section_reform}). Furthermore, we formulate a version of the Section Conjecture in arbitrary characteristic (Conjecture~\ref{conj:section_extend}).

The second quotient is the largest quotient of $\univ X$ that is \emph{tame}, in the sense of \cite{dan-olsson-vistoli1}. Of course in characteristic~0 we have $\univ X = \etuniv X = \tuniv X$.

It is natural to ask what are the tannakian subcategories of $\efin X$ that correspond to $\etuniv X$ and $\tuniv X$. The first one does not seem to have a natural answer (although there is a tannakian interpretation of $\etuniv X$ in terms of local systems, see Section~\ref{sec:etale-tannakian}). 

In Section~\ref{sec:tame-tannakian} we show that the category of representations of $\tuniv X$ is equivalent to the category of \emph{finite tame locally free sheaves} on $X$: these are the finite locally free sheaves $E$ on $X$ such that all tensor powers $E^{\otimes n}$ are semisimple.

Finally, Section~\ref{sec:examples} contains some examples and applications to illustrate the theory.

\subsection*{Acknowledgments} The authors would like to thank Bertrand Töen for an extremely stimulating conversation. They are also grateful to Michel Emsalem, Damian Rössler, and Jakob Stix, for several discussions, especially about Section~\ref{sec:section}, and to the referee for his useful comments. The content of Section~\ref{sec:examples} started out as joint work with Sylvain Brochard, whom we thank heartily. Finally, we also thank the ENS Lyon for its hospitality while much of this research was done.

\section{Conventions}\label{sec:conventions}

We will work over a fixed field $\kappa$; all schemes and all morphisms of schemes will be over $\kappa$, unless explicit mention to the contrary is made. All fibered categories will be fibered in groupoids over the category $\aff{\kappa}$ of affine schemes over $\kappa$, with the same proviso. As usual, we will identify a functor $\aff\kappa\op \arr \catset$ with the corresponding fibered category, and a scheme with the corresponding functor.  A fibered product (of fibered categories, or of schemes) $X \times_{\spec \kappa}Y$ will be denoted simply by $X \times Y$. If $\kappa'$ is an extension of $\kappa$ and $X$ is a fibered category over $\kappa$, we set $X_{\kappa'} \eqdef \aff {\kappa'} \times_{\aff \kappa}X$.

If $p\colon X \arr \aff \kappa$ is a fibered category, we will consider it as a site by putting on it the fpqc topology, generated by the collection of morphisms $\{\xi_{i} \arr \xi\}$ such that $\{p(\xi_{i}) \arr p(\xi)\}$ is an fpqc cover. There is a sheaf $\cO_{X}$ on $X$, sending each object $\xi$ into $\cO\bigl(p(\xi)\bigr)$.

Let $f\colon X \arr Y$ be a morphism of fibered categories, where $Y$ is an algebraic stack. We call the \emph{scheme-theoretic image} $Y'$ of $X$ in $Y$ the intersection of all the closed substacks $Z$ of $Y$ such that $f$ factors, necessarily uniquely, though $Z$. Alternatively, $Y' \subseteq Y$ is the closed substack associated with the largest quasi-coherent sheaf of ideals of $\cO_{Y}$ contained in the kernel of the natural homomorphism $\cO_{Y} \arr f_{*}\cO_{X}$. It is easy to see that $f$ factors uniquely through $Y'$ (this is clear when $X$ is a scheme, and the general case reduces to this).

If $X$ is an algebraic stack over $\kappa$, we will call $\vect X$ the category of locally free sheaves of $\cO_{X}$-modules of finite constant rank on $X$, which we will also call \emph{vector bundles}. If $X$ is a locally noetherian algebraic stack, we denote by $\coh X$ the category of coherent sheaves on $X$. We also denote by $\vect_{\kappa}$ the category $\vect(\spec \kappa)$ of vector spaces on $\kappa$.

All 2-categories appearing in this paper will be strict $(2,1)$-categories, unless we mention otherwise. Likewise, all functors will be strict. We will use the symbol $*$ for the Godement product.

\section{Projective limits of fpqc gerbes}\label{sec:projective-limits}

We will work with the 2-category $\affger$ of affine fpqc gerbes over $\kappa$, called \emph{tannakian gerbes} in \cite[Chapitre~III, \S 2]{saavedra}. These are fpqc gerbes with a flat presentation $R \double U$, where $R$ and $U$ are affine $\kappa$-schemes. Equivalently they can be defined as fpqc gerbes over $\catsch{\kappa}$ with affine diagonal, and an affine chart.

Suppose that $R \double U$ is a flat groupoid, where $R$ and $U$ are affine; then the associated fpqc stack is a gerbe if and only if the diagonal $R \arr U \times U$ is faithfully flat, and $U$ is nonempty.

If an affine fpqc gerbe $\Phi$ has an object $\xi$ defined over $\kappa$, then it is equivalent to the classifying $\cB_{k}G$, where $G$ is the automorphism group scheme of $\xi$ over $\kappa$. 

\begin{proposition}\call{prop:properties-fpqc-gerbe}

Let $\Phi$ be an affine fpqc gerbe.

\begin{enumeratea}

\itemref{a} $\Phi$ has an fpqc presentation of the type $R \double U$, where $R$ is affine and $U$ is the spectrum of a field.

\itemref{b} Any morphism from a non-empty algebraic stack $X$ to $\Phi$ is faithfully flat, and an fpqc cover. If $X$ is a scheme, it is moreover 
	representable.

\itemref{c} The diagonal $\Phi \arr \Phi \times \Phi$ is representable, faithfully flat, and affine.

\end{enumeratea}

\end{proposition}

\begin{proof}
For part \refpart{prop:properties-fpqc-gerbe}{b}, one can make a field extension; so we can assume that $\Phi(\kappa) \neq \emptyset$, so that $\Phi \simeq \cB_{\kappa}G$, where $G$ is a affine group scheme. Then the morphism $\spec \kappa \arr \cB_{\kappa}G$ corresponding to the trivial torsor $G \arr \spec \kappa$ is the universal torsor over $\cB_{\kappa}G$; in particular, it is affine and faithfully flat, hence an fpqc cover. Hence, if $S$ is a non-empty scheme and $S \arr \cB_{\kappa}G$ is a morphism, the fibered product $P \eqdef S \times_{\cB_{\kappa}G} \spec \kappa$ is a $G$-torsor over $S$, so it is non-empty, and $P \arr \spec \kappa$ is an fpqc cover. Hence $S \arr \cB_{\kappa}G$ is an fpqc cover, as claimed.

For part~\refpart{prop:properties-fpqc-gerbe}{c} we can also make an extension of base field, since affine morphisms satisfy fpqc descent, so that $\Phi(\kappa) \neq \emptyset$. In this case $\Phi \simeq \cB_{\kappa}G$ for an affine group $G$, and the statement is clear.

For \refpart{prop:properties-fpqc-gerbe}{a}, take a field extension $K$ of $\kappa$ such that $\Phi(K) \neq \emptyset$, and set $U \eqdef \spec K$ and $R \eqdef U \times_{\Phi} U$. Since the diagonal $\Phi \arr \Phi \times \Phi$ is affine, we have that $R$ is an affine scheme. Because of \refpart{prop:properties-fpqc-gerbe}{b}, the groupoid $R \double U$ gives an fpqc presentation of $\Phi$.
\end{proof}

\begin{definition}
A \emph{boolean cofiltered $2$-category} $I$ is a small 2-category such that, for any two objects $i$ and $j$ of $I$,
\begin{enumeratea}

\item there exists another object $k$ with arrows $k \arr i$ and $k \arr j$, and

\item given any two 1-arrows $a$, $b\colon j \arr i$, there exists a unique 2-arrow $a \darr b$.

\end{enumeratea}
\end{definition}

One sees immediately that a $2$-category is cofiltered if and only if it is equivalent to a cofiltered partially ordered set, considered as a 2-category.

From now we will call a boolean cofiltered $2$-category simply a \emph{cofiltered $2$-category}. The reason for the adjective ``boolean'' is that, as was pointed out by the referee, a general cofiltered $1$-category is not equivalent to a partially ordered set, and so it can not be a boolean cofiltered $2$-category.

\begin{definition}
A \emph{projective system} in $\affger$ consists of a cofiltered 2-category $I$ and a strict 2-functor $\Gamma\colon I \arr \affger$.
\end{definition}

\begin{remark}
There are two possible ways of defining a projective system. One as in the definition above; the other is to take $I$ to be a cofiltered partially ordered set, and $\Gamma$ to be a pseudo-functor. The two methods are essentially equivalent, but the one above works better for our purposes.
\end{remark}

Given a projective system $\Gamma\colon I \arr \affger$, we denote by $\Gamma_{i}$ the image of an object $i$ of $I$, and by $\Gamma_{a}\colon \Gamma_{j} \arr \Gamma_{i}$ the cartesian functor corresponding to a 1-arrow $a\colon j \arr i$. Finally, if $a$, $b\colon j \arr i$ are 1-arrows, we denote by $\Gamma_{a,b}\colon \Gamma_{a} \arr \Gamma_{b}$ the natural isomorphism corresponding to the unique 2-arrow $a \to b$.

\begin{definition}
Let  $\Gamma\colon I \arr \affger$ be a projective system of affine gerbes. The \emph{projective limit} $\projlim \Gamma$ is the category fibered in groupoids over $\catsch{\kappa}$ defined as follows.

An object $\xi$ of $\projlim \Gamma$ consists of the following data.

\begin{enumerate1}

\item A scheme $T$ over $\kappa$, and an object $\xi_{i}$ of $\Gamma_{i}(T)$ for each object $i$ of $\Gamma$.

\item For each 1-arrow $a\colon j \arr i$ in $I$, an arrow $\xi_{a}\colon \Gamma_{a}(\xi_{j}) \arr \xi_{i}$ in $\Gamma_{i}(T)$.

\end{enumerate1}

These are required to satisfy the following conditions.

\begin{enumeratea}

\item If $a\colon j \arr i$ and $b\colon k \arr j$ are 1-arrows in $I$, the diagram
   \[
   \xymatrix@C+15pt{
   \Gamma_{ab}(\xi_{k})
   \ar[r]^{\Gamma_{a}(\xi_{b})}\ar[rd]_{\xi_{ab}}
   &\Gamma_{a}(\xi_{j}) \ar[d]^{\xi_{a}}\\
   &\xi_{i}
   }
   \]
commutes.

\item If $a$ and $b$ are 1-arrows from $j$ to $i$, the diagram
   \[
   \xymatrix{
   \Gamma_{a}(\xi_{j}) \ar[rr]^{\Gamma_{a,b}}\ar[rd]_{\xi_{a}}
      && \Gamma_{b}(\xi_{j})\ar[ld]^{\xi_{b}}\\
   & \xi_{i}
   }
   \]
commutes.
\end{enumeratea}

An arrow $f\colon  \xi \arr \eta$ consists of a morphism of $\kappa$-schemes $\phi\colon T \arr T'$ and an arrow $f_{i}\colon \xi_{i} \arr \eta_{i}$ for each $i$, such that
\begin{enumeratea}

\item $f_{i}$ maps to $\phi$ in $\catsch k$ for all $i$, and

\item for each 2-arrow $a\colon j \arr i$, the diagram
   \[
   \xymatrix@C+15pt{
   \Gamma_{a}(\xi_{j}) \ar[r]^{\Gamma_{a}(f_{j})} \ar[d]^{\xi_{a}}
   &\Gamma_{a}(\eta_{j})\ar[d]^{\eta_{a}}\\
   \xi_{i} \ar[r]^{f_{i}} & \eta_{i}
   }
   \]
commutes.
\end{enumeratea}
\end{definition}

For the projective limit $\projlim \Gamma$ we will also use the notation $\projlim_{i \in I} \Gamma_{i}$, or $\projlim_{i} \Gamma_{i}$.

It is a straightforward exercise in descent theory to show that $\projlim \Gamma$ is an fpqc stack.

\begin{remark}\label{rmk:viewpoint-groupoids}
Here is different description of $\projlim_{i \in I} \Gamma_{i}$ from the point of view of groupoids. Suppose that $\Phi \eqdef \projlim \Gamma$ is non-empty (it is not clear to us whether it can happen that it is empty); let $X$ be a non-empty scheme with a morphism $X \arr \projlim \Gamma$, corresponding to an object $\xi= \{\xi_{i}\}$ of $\Phi(X)$. We claim that $X \arr \Phi$ is representable, faithfully flat, and an fpqc cover.

For this, let $T$ be an affine scheme with a morphism $T \arr \Phi$, corresponding to an object $\tau = \{\tau_{i}\}$ of $\Phi(T)$. The fibered product $X \times_{\Gamma_{i}}T$ is represented by the scheme $P_{i} \eqdef \underisom_{X \times T}(\pr_{2}^{*}\tau_{i}, \pr_{1}^{*}\xi_{i})$; by Proposition \refall{prop:properties-fpqc-gerbe}{c}, $P_{i}$ is affine and faithfully flat over $X \times T$. If $a\colon j \arr i$ is an arrow in $I$, the corresponding isomorphisms $\xi_{a}\colon \Gamma_{a}(\xi_{j}) \arr \xi_{i}$ and $\tau_{a}\colon \Gamma_{a}(\tau_{j}) \arr \tau_{i}$ induce a morphism of $X \times T$-schemes $P_{j} \arr P_{i}$; this defines a functor from $I$ to the category $\aff{X \times T}$ of affine schemes over $X \times T$. Since the target is a 1-category, this functor factors through  the preordered set $\overline{I}$ corresponding to $I$ (that is, the objects are the objects of $I$, and we set $j \leq i$ when there exists an arrow $j \arr i$). Set $P \eqdef \projlim_{\overline{I}}P_{i}$; it follows from the definition of projective limit that $P$ represents the functor $\underisom_{X \times T}(\pr_{2}^{*}\tau, \pr_{1}^{*}\xi) = X \times_{\Phi} T$. Since the limit of affine faithfully flat schemes over $X \times T$ is again affine and faithfully flat, we see that $P$ is affine and faithfully flat over $X \times T$; hence $P$ is an fpqc cover of $T$, and the result follows.

If we set $R_{i} = X \times_{\Gamma_{i}} X$ and $R = \projlim R_{i}$, it results from the above that $\projlim \Gamma$ is the fpqc quotient of the groupoid $R \double X$; this will be used in Remark~\ref{rmk:compare-Deligne} to compare our construction with Deligne's construction of the étale fundamental groupoid.

Also, since $R \arr X \times X$ is faithfully flat and affine, and $X$ is nonempty, we have that $\Phi$ is a fpqc gerbe. We record this in a Proposition.
\end{remark}

\begin{proposition}\label{prop:limit-gerbes}
Let $\Gamma\colon I \arr \affger$ a projective system of affine fpqc gerbes. If the limit $\projlim \Gamma$ is not empty, it is an fpqc gerbe.
\end{proposition}

\begin{proposition}\label{prop:factor-through-finite}
Let $\Phi = \projlim_{i}\Gamma_{i}$ be a projective limit of affine fpqc gerbes, $\Delta$ an affine fpqc stack with a flat presentation by affine $\kappa$-schemes of finite type. Then the natural functor $\indlim_{i}\hom(\Gamma_{i}, \Delta) \arr \hom(\Phi, \Delta)$ is an equivalence.
\end{proposition}

\begin{proof}
The stack $\Delta$ is finitely presented over $\kappa$. This implies that if $\{T_{i}\}$ is a projective system of affine $\kappa$-schemes, the natural functor $\indlim_{i}\Delta(T_{i}) \arr \Delta(\projlim T_{i})$ is an equivalence.

Let $S \arr \Phi$ be a morphism, where $S = \spec K$ is the spectrum of a field. Set $R \eqdef S \times_{\Phi}S$ and $R_{i}\eqdef S\times_{\Gamma_{i}} S$ for each $i$; as in the proof of Proposition~\ref{prop:limit-gerbes}, we have $R = \projlim R_{i}$. There is an equivalence of categories between $\hom(\Phi, \Delta)$ and the category of objects of $\Delta(S)$ with descent data on the flat groupoid $R \double S$, and, analogously, an equivalence of $\hom(\Gamma_{i}, \Delta)$ with the category of objects of $\Delta(S)$ with descent data on $R_{i} \double S$. But the equivalence of $\indlim_{i}\Delta(R_{i})$ with $\Delta(R)$ is easily seen to yield an equivalence between the category of objects with descent data on $R \double S$ and the colimit of the categories of objects with descent data on $R_{i} \double S$.
\end{proof}

If $\Phi$ is an fpqc gerbe over $\kappa$, we denote by $\rep \Phi$ the category of coherent sheaves on $\Phi$; these are all locally free, because $\Phi$ admits a faithfully flat morphism $\spec K \arr \Phi$, where $K$ is a field. So $\rep \Phi = \vect \Phi$. The category $\rep \Phi$ is tannakian.

\begin{proposition}
Suppose that $\Phi = \projlim_{i}\Gamma_{i}$ is a non-empty limit of fpqc gerbes. The natural functor $\indlim_{i} \rep \Gamma_{i} \arr \rep \Phi$ is an equivalence.
\end{proposition}

\begin{proof}
Let $\Delta$ be the fibered category of locally free sheaves on $\catsch \kappa$, that is, the fibered category whose objects are pairs $(T, E)$, where $T$ is a $\kappa$-scheme and $E$ is a locally free sheaf on $T$, and in which the arrows are given by arbitrary homomorphism of locally free sheaves (thus, $\Delta$ is not fibered in groupoids). Then $\Delta$ is still finitely presented, in the sense that the natural functor $\indlim_{i}\Delta(T_{i}) \arr \Delta(\projlim T_{i})$ is an equivalence for any projective system of affine $\kappa$-schemes $\{T_{i}\}$. Hence, the proof of Proposition~\ref{prop:factor-through-finite} goes through in this case.
\end{proof}

\section{Finite stacks}\label{sec:finite-stacks}

\begin{definition}
A \emph{finite stack} over $\kappa$ is an fppf stack that is represented by a flat groupoid $R \double X$, where $R$ and $X$ are finite $\kappa$-schemes.

A \emph{finite gerbe} is a finite stack over $\kappa$ that is a gerbe in the fppf topology.
\end{definition}

By a well known theorem of M.~Artin \cite[Théorème~(10.1)]{laumon-moret-bailly}, a finite stack is algebraic; it can be defined as an algebraic stack $\Gamma$ over $\kappa$ with finite diagonal, which admits a flat surjective map $U \arr \Gamma$, where $U$ is finite over $\kappa$. Notice that a finite stack is always an fpqc stack, as it can be interpreted as the stack of $R \double X$-torsors (if $R\double X$ is an fppf groupoid, then an fpqc torsor is also an fppf torsor, as it is trivialized by one the projections $R \arr X$).

We have the following useful characterization of finite stacks.

\begin{proposition}\label{prop:char-finite}
Let $\Gamma$ be an algebraic stack over $\kappa$. Then $\Gamma$ is finite stack if and only if the following conditions are satisfied.

\begin{enumeratea}

\item $\Gamma$ is of finite type over $\kappa$.

\item The diagonal of $\Gamma$ is quasi-finite.

\item The category $\Gamma(\overline{\kappa})$ has finitely many isomorphism classes.

\end{enumeratea}
\end{proposition}

\begin{proof}
It is immediate to check that if $\Gamma$ is finite, it has the properties above. Conversely, assume that the conditions are satisfied.

If $\Gamma$ is an algebraic space of finite type over $k$, and $\Gamma(\overline{k})$ is finite, then it is immediate to see that $\Gamma$ is in fact the spectrum of a finite $k$-algebra. Let us reduce the general case to this one. By \cite[V.7]{sga3-1} (see also the second paragraph of the proof of \cite[Proposition~2.11]{olsson-hom-stacks}) there exists a scheme $X$ and a quasi-finite flat surjective morphism $X \arr \Gamma$ of finite type; then $\Gamma$ has an fppf presentation $X\times_{\Gamma}X \double X$.

Since $X(\overline{\kappa})$ is immediately seen to be finite, it follows that $X$ is finite; and since $X\times_{\Gamma}X \arr X \times X$ is quasi-finite, $(X\times_{\Gamma}X)(\overline{k})$ is also finite, so $X\times_{\Gamma}X$ is finite. The result follows.

\end{proof}

\begin{proposition}\label{prop:gerbe<->geom}
A finite stack over $\kappa$ is a finite gerbe if and only if it is \geom.
\end{proposition}

\begin{proof}
It is obvious that a finite gerbe is connected; it is also reduced, since it admits a faithfully flat map from the spectrum of a field (Proposition~\refall{prop:properties-fpqc-gerbe}{a}). Since by extending the base field a finite gerbe stays a finite gerbe, it is also \geom.

Conversely, suppose that $\Gamma$ is a finite stack that is \geom. Being an fppf gerbe is a local property in the fppf topology, so we may base change to a finite extension of $\kappa$, and assume that $\Gamma$ admits a section $\spec \kappa \arr \Gamma$, corresponding to an object $\xi$ of $\Gamma(\kappa)$. The fibered product $G \eqdef \spec \kappa \times_{\Gamma}\spec \kappa$ is the group scheme of automorphisms of $\xi$; if we show that $\spec \kappa \arr \Gamma$ is flat and surjective, then $\Gamma$ will have a presentation $G \double \spec \kappa$, that is, $\Gamma \simeq \cB_{\kappa}G$, so it is indeed an fppf gerbe.

The morphism $\spec \kappa \arr \Gamma$ is of finite type. By the theorem on generic flatness, it is generically flat, since $\Gamma$ is connected; but since $\spec \kappa$ consists of a point, it is in fact flat. Hence its image in $\Gamma$ is open. Since $\Gamma$ is finite, it is also closed, so it is surjective, because $\Gamma$ is connected.
\end{proof}

If $G$ and $H$ are finite group schemes over $\kappa$, every homomorphism of group schemes $G \arr H$ induces a morphism of gerbes $\cB_{\kappa}G \arr \cB_{\kappa}H$. However, the category $\hom_{\kappa}(\cB_{\kappa}G, \cB_{\kappa}H)$ is not equivalent to the set $\hom_{\kappa}(G,H)$. Denote by $\underhom_{\kappa}(G, H)$ the fibered category over $\kappa$ of homomorphisms $G \arr H$, that is, the sheaf sending each $\kappa$-scheme $T$ into the set $\hom_{T}(G_{T}, H_{T})$. There is a natural action by conjugation of $H$ on $\underhom_{\kappa}(G, H)$; then the fibered category $\underhom_{\kappa}(\cB_{\kappa}G, \cB_{\kappa}H)$ is equivalent to the category of $\kappa$-valued objects of the quotient stack $[\underhom_{\kappa}(G, H)/H]$ (see \cite[III Remarque~1.6.7]{giraud}).

\begin{definition}
The \emph{degree} $\deg \Gamma$ of a finite gerbe $\Gamma$ is the degree of the automorphism group scheme $\underaut_{\kappa'}\xi$, where $\kappa'$ is an extension of $\kappa$ and $\xi$ is an object of $\Gamma(\spec\kappa')$.
\end{definition}

In the definition above, it is straightforward to show that $\deg \Gamma$ does not depend on $\kappa'$ nor on $\xi$. Furthermore, one should notice that the degree in this sense is not the degree of the proper quasi-finite morphism $\Gamma \arr \spec\kappa$, which equals $1/\deg \Gamma$.

\begin{proposition}\label{prop:properties-degree}
Let $f\colon \Gamma \arr \Delta$ be a representable morphism of finite gerbes. Then $\deg \Gamma$ divides $\deg \Delta$, and $f$ is an isomorphism if and only if $\deg \Gamma = \deg \Delta$.
\end{proposition}

\begin{proof}
The property of being representable is stable under extension of base field, as is the degree; hence we may assume that $\Gamma(\spec\kappa)$ is not empty. Then $f\colon \Gamma \arr \Delta$ is equivalent to a morphism $\phi\colon \cB_{\kappa}G \arr \cB_{\kappa}H$ induced by a homomorphism $G \arr H$ of finite $\kappa$-group schemes. If $f$ is representable, then $\phi$ is injective, and the result is standard in this case.
\end{proof}

\begin{definition}
A \emph{profinite gerbe} over $\kappa$ is an fpqc gerbe that is equivalent to a projective limit of finite gerbes over $\kappa$.
\end{definition}

\section{The Nori fundamental gerbe}\label{sec:fund-gerbe}

Let $X$ be a fibered category over $\kappa$.

\begin{definition}\label{def:fund-gerbe}
A \emph{fundamental gerbe} for $X$ is a profinite gerbe $\univ X$ with a morphism $X \arr \univ X$ of algebraic stacks over $\spec\kappa$, such that, if $\Gamma$ is a finite stack over $\spec\kappa$, the induced functor
   \[
   \hom_{\kappa}(\univ X, \Gamma) \arr \hom_{\kappa}(X, \Gamma)
   \]
is an equivalence of categories.
\end{definition}

\begin{proposition}
	\label{prop:fund-gerbe-unique}
Let $X \arr \univ X$ be a fundamental gerbe. Then for any profinite gerbe $\Phi$ the induced functor
   \[
   \hom_{\kappa}(\univ X, \Phi) \arr \hom_{\kappa}(X, \Phi)
   \]
is equivalence of categories.

In particular, a fundamental gerbe is unique up to a canonical equivalence.
\end{proposition}

\begin{proof}
This follows easily from the definitions of a profinite gerbe and of a projective limit.
\end{proof}

\begin{definition}
We say that a fibered category $X$ is \emph{\infl} if it is non-empty, and for any morphism $X \arr \Gamma$ to a finite stack, there exists a closed substack $\Gamma' \subseteq \Gamma$ that is a gerbe, and a factorization $X \arr \Gamma' \arr \Gamma$.
\end{definition}

Notice that in the statement above, if $\Gamma'$ exists then it must be the scheme-theoretic image of $X$ in $\Gamma$; hence it is unique.

Here are some properties of this notion.

\begin{proposition}\call{prop:properties-infl}
Let $X$ be a fibered category over $\kappa$.

\begin{enumeratea}

\itemref{a} If $X$ is inflexible, the only $\kappa$-subalgebra of $\H^{0}(X, \cO)$ that is finite over $\kappa$ is $\kappa$ itself.

\itemref{b} If $X \neq \emptyset$ and every morphism $X \arr \Gamma$ to a finite stack $\Gamma$ factors through an affine fpqc gerbe, then $X$ is inflexible. In particular, an affine fpqc gerbe is \infl.

\end{enumeratea}
\end{proposition}

\begin{proof}
\refpart{prop:properties-infl}{a}: Let $A$ be a finite $\kappa$-subalgebra of $\H^{0}(X, \cO)$. Then $\spec A$ is a finite stack, so corresponding morphism $X \arr \spec A$ factors trough a closed subgerbe of $\spec A$. But the only scheme over $\kappa$ that is a gerbe is $\spec \kappa$; hence the embedding $A \subseteq \H^{0}(X, \cO)$ factors through $\kappa$, and $A = \kappa$.

\refpart{prop:properties-infl}{b}: Let $f\colon X \arr \Gamma$ be a morphism to a finite stack, and suppose that there is a factorization $X \xarr{g} \Phi \xarr{h} \Gamma$, where $\Phi$ is an affine fpqc gerbe. The adjunction homomorphism $\cO_{\Phi} \arr g_{*}\cO_{X}$ is injective, because $g$ is faithfully flat, by \refall{prop:properties-fpqc-gerbe}{b}. Hence the kernel of $\cO_{\Gamma} \arr f_{*}\cO_{X}$ coincides with the kernel of $\cO_{\Gamma} \arr h_{*}\cO_{\Phi}$. So $X$ and $\Phi$ have the same scheme-theoretic image in $\Gamma$, and we may assume $X = \Phi$. 

By Proposition \refall{prop:properties-fpqc-gerbe}{a}, there exists an extension $K$ of $\kappa$ and an affine faithfully flat morphism $\spec K \arr \Phi$; the scheme-theoretic images of $\Phi$ and of $\spec K$ are evidently the same. Since $\spec K$ is reduced and connected, we deduce that the scheme-theoretic image of $\Phi$ in $\Gamma$ is reduced and connected. On the other hand, the formation of the scheme-theoretic image commutes with extensions of the base field $\kappa$, by Remark~\ref{rmk:base-change} below, and by extending the base field an affine fpqc gerbe remains an affine fpqc gerbe; hence the scheme theoretic image is \geom, and we conclude by Proposition~\ref{prop:gerbe<->geom}.
\end{proof}

This strange condition of being inflexible has some geometric content.

\begin{proposition}\call{prop:geometric-properties-infl}
 Suppose that $X$ is an algebraic stack of finite type over $\kappa$.
\begin{enumeratea}

\itemref{a} If $X$ is \infl, it is geometrically connected.

\itemref{b} If $X$ is \geom, it is \infl.

\end{enumeratea}
\end{proposition}

\begin{proof}
\refpart{prop:geometric-properties-infl}{a}: Suppose that $X$ is \infl, but not geometrically connected; then there exists a finite separable extension $\kappa'$ of $\kappa$ and a surjective morphism of $\kappa'$-schemes $X_{\kappa'} \arr \spec \kappa' \sqcup \spec \kappa'$. Let $\Gamma$ be the $\kappa$-scheme obtained by Weil restriction from $\spec \kappa' \sqcup \spec \kappa'$; it is well known, and easy to see, that the Weil restriction of a finite $\kappa'$-scheme along a finite separable field extension is finite. The morphism $X \arr \Gamma$ corresponding to the given morphism $X_{\kappa'} \arr \spec \kappa' \sqcup \spec \kappa'$ factors through $\spec \kappa$; hence $X_{\kappa'} \arr \spec \kappa' \sqcup \spec \kappa'$ factors through $\spec \kappa'$, by Proposition \refall{prop:properties-infl}{a}, and this is a contradiction.

\refpart{prop:geometric-properties-infl}{b}: Let $\Gamma'$ be the scheme-theoretic image of $f\colon X\to \Gamma$, where $\Gamma$ is a finite stack. The stack $\Gamma'$ is geometrically connected, since $X$ is. We claim that $\Gamma'$ is geometrically reduced. If $\cha \kappa = 0$, then $\Gamma'$ is étale. If $\cha \kappa > 0$, the canonical morphism $\cO_{\Gamma'} \arr f_{*}\cO_{X}$ is injective, by definition; and this injectivity is preserved under finite field extension. 

Hence in any case $\Gamma'$ is \geom, so by Proposition~\ref{prop:gerbe<->geom} it is a gerbe.
\end{proof}

The following examples seem to show that the geometric content in the condition of being \infl is somewhat subtle. A geometric characterization of \infl algebraic stacks eludes us.

\begin{examples}\hfil

\begin{enumeratea}

\item The condition of Proposition \refall{prop:properties-infl}{a} is not sufficient for an algebraic stack, or even a scheme, of finite type to be \infl. For example, let $X_{0}$ be a \geom projective scheme over $\kappa$ with a non-trivial invertible sheaf $L$ and an isomorphism $L^{\otimes 2} \simeq \cO_{X_{0}}$. Consider the relative spectrum $X$ of the finite sheaf of algebras $\cO_{X_{0}} \oplus L$ over $X_{0}$, where the product of two sections of $L$ is always $0$. Then $\H^{0}(X, \cO) = \kappa$, but we claim that $X$ is not \infl.

The invertible sheaf $L$ correspond to a $\mmu_{2}$-torsor $Y_{0} \arr X_{0}$; call $Y$ the relative scheme of the sheaf $\cO_{Y_{0}} \oplus \cO_{Y_{0}}\epsilon$, with $\epsilon^{2} = 0$. There is a free action of $\mmu_{2}$ on $Y$, extending the given action on $Y_{0} \subseteq Y$, changing the sign of $\epsilon$. There is a tautological $\mmu_{2}$-equivariant morphism $Y \arr \spec \kappa[\epsilon]$, where $\kappa[\epsilon]$ is the ring of dual numbers, and $\mmu_{2}$ acts on $\kappa[\epsilon]$ by changing the sign of $\epsilon$. This induces a morphism $X \arr \bigl[\spec k[\epsilon] / \mmu_{2}\bigr]$, which does not factor through a gerbe.

\item On the other hand, there are examples of projective schemes over $\kappa$ that are \infl without being reduced.

Suppose that $X_{0}$ is a \geom positive-dimensional projective scheme over $\kappa$, and let $L$ be the dual of an ample invertible sheaf on $X_{0}$. Consider the relative spectrum $X$ of the finite sheaf of algebras $\cO_{X_{0}} \oplus L$ over $X_{0}$, where the product of two sections of $L$ is always $0$. Assume also the characteristic of $\kappa$ is $0$ (this is not necessary, but makes the proof somewhat easier). Then we claim that $X$ is \infl.

The scheme $X_{0}$ is a closed subscheme of $X$ with sheaf of ideals $L$. Take a morphism $f\colon X \arr \Gamma$ to a finite stack. Assume that the homomorphism $\cO_{\Gamma} \arr f_{*}\cO_{X}$ is injective (that is, the scheme-theoretic image of $X$ in $\Gamma$ is $\Gamma$ itself); we need to show that $\Gamma$ is a gerbe. Call $\Gamma_{0}$ the reduced substack of $\Gamma$, and $N$ the sheaf of ideals of $\Gamma_{0}$ in $\Gamma$. We have $N^{2} = 0$, because $N^{2}$ pulls back to $L^{2} = 0$. The morphism $f$ restricts to a morphism $f_{0}\colon X_{0} \arr \Gamma_{0}$ with scheme-theoretic image $\Gamma_{0}$; hence $\Gamma_{0}$ is a gerbe, since $X_{0}$ is \infl; so we need to show that $N = 0$.

The pullback $f^{*}N = f_{0}^{*}N$ maps to $L$, and it is enough to show that this map is $0$. Choose a morphism $\phi\colon \spec K \arr \Gamma_{0}$, where $K$ is a finite extension of $\kappa$; this map is étale, since we are in characteristic~$0$. Consider the cartesian diagram
   \[
   \xymatrix{
   Y \ar[r]^-{g}\ar[d]^{\psi} & \spec K \ar[d]^{\phi}\\
   X_{0} \ar[r]^{f_{0}} & \Gamma_{0}\hsmash{\,.}
   }
   \]
It is enough to show that $\H^{0}(X_{0}, f^{*}N^{\vee} \otimes L) = 0$; this follows if we prove that 
   \[
   \H^{0}\bigl(Y, \psi^{*}(f^{*}N^{\vee} \otimes L)\bigr) =  
   \H^{0}(Y, g^{*}\phi^{*}N^{\vee} \otimes \psi^{*}L))
   \]
is $0$. But this is clear, because $\phi^{*}N^{\vee}$ is free, the sheaf $\psi^{*}L$ is the dual of an ample invertible sheaf on $Y$, and $Y$ is étale over $X_{0}$, hence projective and reduced.
\end{enumeratea}
\end{examples}

The following is the main result of this section.

\begin{theorem}\label{thm:universal-exists}
A fibered category over $\kappa$ has a fundamental gerbe if and only if it is \infl.
\end{theorem}

From this and Proposition \refall{prop:geometric-properties-infl}{b} we obtain the following.

\begin{corollary}
A \geom algebraic stack of finite type over $k$ has a fundamental gerbe.
\end{corollary}

\begin{remark}
Of course one could relax the definition of fundamental gerbe and only require that it be universal for maps from $X$ to finite (or, equivalently, profinite) gerbes. However, still $\univ X$ would not exist in general. For example, it is easy to see that $\spec\kappa \sqcup \spec \kappa$ can not have a universal gerbe in this sense.

When $\cha \kappa = 0$, one can show that an algebraic stack of finite type has a fundamental gerbe in this weaker sense if and only if it is geometrically connected. The main point is that in characteristic 0 every finite gerbe is étale, so for any finite gerbe $\Gamma$ the restriction functor $\hom(X, \Gamma) \arr \hom(X_{\mathrm{red}}, \Gamma)$ is an equivalence; hence the fundamental gerbe for $X_{\mathrm{red}}$ is also a fundamental gerbe for $X$.

However, in positive characteristic the exact conditions for the existence of a fundamental gerbe in this weaker sense are not clear.
\end{remark}

Let us prove Theorem~\ref{thm:universal-exists}. First, assume that $X$ has a fundamental gerbe $\univ X$, and take a morphism $X \arr \Gamma$ to a finite stack. This factors through $\univ X$, and the result follows from \refall{prop:properties-infl}{b}.

Now assume that $X$ is \infl.

\begin{definition}
Let $\Gamma$ be a finite gerbe. A morphism of fibered categories $X \arr \Gamma$ is \emph{Nori-reduced} if for any factorization $X \arr \Gamma' \arr \Gamma$, where $\Gamma'$ is a finite gerbe and $\Gamma' \arr \Gamma$ is faithful, then $\Gamma' \arr \Gamma$ is an isomorphism.
\end{definition}

\begin{remark}
The notion of Nori-reduced morphism is perhaps clarified by the following fact. Suppose that $G$ is a finite étale group scheme over $\kappa$ and $X \arr \cB_{\kappa}G$ is a morphism, where $X$ is a \geom stack of finite type over $\kappa$. This morphism corresponds to a $G$-torsor $Y \arr X$. We claim that $X \arr \cB_{\kappa}G$ is Nori-reduced if and only if $Y$ is geometrically connected.

In fact, a representable map $\Gamma \arr \cB_{\kappa}G$, where $\Gamma$ is a finite gerbe, is given by the projection
$[U/G] \arr \cB_{\kappa}G$, where $U$ is a finite étale scheme on which $G$ acts transitively. So, $X \arr \cB_{\kappa}G$ is not Nori reduced if and only if there exists a finite étale scheme $U$, different from $\spec \kappa$, on which $G$ acts transitively, and a $G$-equivariant morphism $Y \arr U$. If this exists, $Y$ cannot by geometrically connected, since $Y \arr U$ must be surjective. Conversely, if $Y$ is not geometrically connected, you can take as $U$ the spectrum of the algebraic closure of $\kappa$ in $\H^0(Y, \cO)$, with the natural action of $G$.

\end{remark}

\begin{lemma}\label{lem:factor-through-Nori-reduced}
Let $\Gamma$ be a finite gerbe, $X$ an \infl fibered category, $X \arr \Gamma$ a morphism. Then there exists a factorization $X \arr \Delta \arr \Gamma$, where $\Delta$ is a finite gerbe, $\Delta \arr \Gamma$ is representable, and $X \arr \Delta$ is Nori-reduced.

Furthermore, $X \arr \Delta$ is unique up to equivalence.
\end{lemma}

\begin{proof}
Take a factorization $X \arr \Delta \arr \Gamma$ with $\Delta \arr \Gamma$ representable, and suppose that $X \arr \Delta$ is not Nori-reduced. By definition there will be a factorization $X \arr \Delta' \arr \Delta$, where $\Delta' \arr \Delta$ is representable, but not an isomorphism. By Proposition~\ref{prop:properties-degree}, the degree of $\Delta'$ is less than the degree of $\Delta$. The proof is concluded by induction on the degree of $\Delta$.

For the second part, suppose that $X \arr \Delta \arr \Gamma$ and $X \arr \Delta' \arr \Gamma$ are two factorizations as in the statement. Then $\Delta\times_{\Gamma}\Delta'$ is a finite stack, and its two projections onto $\Delta$ and $\Delta'$ are representable. Consider the morphism $X \arr \Delta \times_{\Gamma} \Delta'$ induced by the two morphism above; since $X$ is \infl, the scheme-theoretic image $\Delta''$ of $X$ in $\Delta\times_{\Gamma}\Delta'$ is a finite gerbe, and the two morphisms $\Delta'' \arr \Delta$ and $\Delta'' \arr \Delta'$ are representable. Since $X \arr \Delta$ and $X \arr \Delta'$ are Nori-reduced, $\Delta'' \arr \Delta$ and $\Delta'' \arr \Delta'$ are equivalences, and the result follows.
\end{proof}

Here is the key lemma.

\begin{lemma}\label{lem:unique-arrow}
Let $f\colon X \arr \Gamma$ and $g\colon X \arr \Delta$ be morphisms of fibered categories, where $\Gamma$ and $\Delta$ are finite gerbes and $f$ is Nori-reduced. Suppose that $u$, $v\colon \Gamma \arr \Delta$ are morphism of fibered categories, and $\alpha\colon u \circ f \simeq g$ and $\beta\colon v \circ f \simeq g$ are isomorphisms. Then there exists a unique isomorphism $\gamma\colon u \simeq v$ such that $ \beta\circ(\gamma * \id_{f})=\alpha  $.
\end{lemma}

This can be expressed by saying that, given two 2-commutative diagrams
   \[
   \xymatrix{
   X \ar[r]^{f}\ar[rd]_{g} &\Gamma \ar[d]^{u}\\
   &\Delta
   }
   \quad
   \text{and}
   \quad
   \xymatrix{
   X \ar[r]^{f}\ar[rd]_{g} &\Gamma \ar[d]^{v}\\
   &\Delta
   }
   \]
in which $f$ is Nori-reduced, there exists a unique isomorphism $u \simeq v$ making the diagram
   \[
   \xymatrix@C+15pt{
   X \ar[r]^{f}\ar[rd]_{g} &\Gamma \ar@/_/[d]_{u}\ar@/^/[d]^{v}\\
   &\Delta   
   }
   \]
2-commutative.

\begin{proof}
Consider the category $\Gamma' \arr \Gamma$ fibered in sets over $\Gamma$, whose objects over a $\kappa$-scheme $T$ are pairs $(\xi, \rho)$, where $\xi$ is an object of $\Gamma(T)$ and $\rho$ is an isomorphism of $u(\xi)$ with $v(\xi)$ in $\Delta(T)$. This can be written as a fibered product
   \[
   \xymatrix{\ar@{}[rd]|{\square}
   \Gamma' \ar[r]\ar[d] & \Gamma \ar[d]^{\generate{u,v}}\\
   \Delta \ar[r] & \Delta \times \Delta\,,
   }
   \]
where the morphism $\Delta \arr \Delta \times \Delta$ is the diagonal. So $\Gamma'$ is a fibered product of finite stacks, hence it is a finite stack.

An isomorphism $u \simeq v$ corresponds to a section of the projection $\Gamma' \arr \Gamma$, or, again, to a substack $\Gamma'' \subseteq \Gamma'$ such that the restriction $\Gamma'' \arr \Gamma$ of the projection is an isomorphism. The composite isomorphism $u \circ f \xarr{\alpha} g \xarr{\beta^{-1}} v \circ f$ yields a lifting $X \arr \Gamma'$ of $f\colon X \arr \Gamma$; the thesis can be translated into the condition that there exists a unique substack $\Gamma'' \subseteq \Gamma'$ as above, such that $X \arr \Gamma'$ factors through $\Gamma''$. Since $X$ is \infl, there is a unique substack $\Gamma''$ of $\Gamma'$ that is a gerbe, such that $X \arr \Gamma'$ factors through $\Gamma''$. However, $\Gamma'' \arr \Gamma$ is representable, because $\Gamma' \arr \Gamma$ is, so $\Gamma'' \arr \Gamma$ is an isomorphism, since $f$ is Nori-reduced.
\end{proof}

\begin{proof}[Proof of Theorem~\ref{thm:universal-exists}]
Consider the 2-category whose objects are Nori-reduced morphisms $X \arr \Gamma$, and whose 1-arrows from $f\colon X \arr \Gamma$ to $g\colon X \arr \Delta$ are pairs $(u, \alpha)$, where $u\colon \Gamma \arr \Delta$ is a morphism of finite stacks and $\alpha\colon u \circ f \simeq g$ is an isomorphism A 2-arrow from $(u, \alpha)$ to $(v, \beta)$ is an isomorphism $\gamma\colon u \simeq v$ such that $\beta\circ (\gamma * \id_{f}) =\alpha  $. The composites and the Godement products are defined in the obvious way. Let $I$ be a skeleton of this category; it is a small 2-category.

We claim that $I$ is a cofiltered 2-category. The fact that given any 1-arrows between two fixed objects there is a unique arrow between them is the content of Lemma~\ref{lem:unique-arrow}. Let us check that given two objects $X \arr \Gamma_{i}$ and $X \arr \Gamma_{j}$, there is an object $X \arr \Gamma_{k}$ with an arrow to both. Since $X$ is \infl, the morphism $X \arr \Gamma_{i} \times \Gamma_{j}$ induced by two objects can be factored through a finite gerbe $\Gamma' \subseteq \Gamma_{i} \times \Gamma_{j}$; from Lemma~\ref{lem:factor-through-Nori-reduced} we see that $X \arr \Gamma'$ can be lifted to a morphism $X \arr \Gamma_{k}$, which is an object in $I$.

We set $\univ X \eqdef \projlim \Gamma$. The morphisms $X \arr \Gamma_{i}$ induce a morphism of fibered categories $X \arr \univ X$. If $\Delta$ is a finite stack over $\kappa$, we need to show that the functor $\hom(\univ X, \Delta) \arr \hom(X, \Delta)$ is an equivalence. It follows from Lemma~\ref{lem:factor-through-Nori-reduced} that it is essentially surjective, so we need to show that it is fully faithful. By Proposition~\ref{prop:factor-through-finite} we see that it is enough to show that given a Nori-reduced morphism $X \arr \Gamma_{i}$, the induced functor $\hom(\Gamma_{i}, \Delta) \arr \hom(X, \Delta)$ is fully faithful. This follows immediately from Lemma~\ref{lem:unique-arrow}.
\end{proof}

\begin{remark}
It was pointed out to us by Bertrand Töen that an alternate proof of Theorem~\ref{thm:universal-exists} can be given along the following lines. Consider the embedding of the 2-category of finite stacks into the category of all algebraic stacks. This embedding preserves 2-limits, hence, it extends to a functor from the 2-category of pro-objects in the category of all algebraic stacks. By a 2-categorical analogue of the adjoint functor theorem, this has a right adjoint, which associates with each algebraic stack a universal pro-object in the 2-category of finite stacks. From the definition of \infl stack, the universal pro-object of an \infl stack is a pro-object in the category of finite gerbes. By the results of Section~\ref{sec:projective-limits}, this can be thought of a profinite gerbe. 

This clarifies considerably the meaning of Theorem~\ref{thm:universal-exists}. Unfortunately we don't know a reference for the 2-categorical result used above. Furthermore, we find the direct construction, via Nori-reduced morphisms, both useful and enlightening.
\end{remark}

\begin{remark}\label{rmk:gerbe-vs-group}
Given an affine group scheme $G$ over $\kappa$, we obtain an fpqc gerbe $\cB_{\kappa}G$, together with a preferred object $G \arr \spec \kappa$ in $\cB_{\kappa}G(\kappa)$, the trivial torsor. Conversely, given an affine fpqc gerbe $\Phi$ and an object $\xi$ of $\Phi(\kappa)$, we obtain an affine group scheme $\underaut_{\kappa}\xi$, with a canonical equivalence $\Phi \simeq \cB_{\kappa}\underaut_{\kappa}\xi$.

Consider the 2-category whose objects are pairs $(\Phi, \xi)$, where $\Phi$ is an fpqc gerbe over $\kappa$ and $\xi$ is an object of $\Phi(\kappa)$. The 1-arrows $(\Phi, \xi) \arr (\Psi, \eta)$ are pairs $(F, a)$, where $F\colon \Phi \arr \Psi$ is a cartesian functor, and $a\colon F(\xi) \simeq \eta$ is an isomorphism in $\Psi(\kappa)$. The 2-arrows $(F, a) \arr (F', a')$ are the base-preserving natural transformations $F \arr F'$, which are compatible with $a$ and $a'$, in the obvious sense. It is easy to see that this 2-category is equivalent to the 1-category of affine group schemes; an affine group scheme $G$ corresponds to the pair $(\cB_{\kappa}G, G \to \spec \kappa)$, while an object $(\Phi, \xi)$ is carried to $\underaut_{\kappa}\xi$. In this correspondence profinite gerbes correspond to profinite group schemes.

Now, assume that $X$ is \infl, and that we are given an object $x_{0}$ of $X(\kappa)$, corresponding to a section $x_{0}\colon \spec \kappa \arr X$. The image $\xi_{0}$ of $x_{0}$ in $\univ X(\kappa)$ gives a profinite group $\pi(X, x_{0}) \eqdef \underaut_{\kappa}\xi_{0}$; we claim that this is the fundamental group scheme in the sense of Nori. This means the following.

Denote by $P_{0} \arr X$ the $\pi(X, x_{0})$-torsor corresponding to the morphism $X \arr \univ X \simeq \cB_{\kappa}\pi(X, x_{0})$; by construction we have a trivialization $x_{0}^{*}P_{0} \simeq \pi(X, x_{0})$. Suppose that we are given a finite group scheme $G$ with a homomorphism $\pi(X, x_{0}) \arr G$; by transport of structure we obtain a $G$-torsor $P \arr X$ with a trivialization $x_{0}^{*}P \simeq G$. 

Conversely, suppose we are given a finite group scheme $G$, a $G$-torsor $P \arr X$, and a trivialization $x_{0}^{*}P \simeq G$. This gives a factorization of the morphism $\spec \kappa \arr \cB_{\kappa}G$ corresponding to the trivial torsor as $\spec \kappa \arr X \arr \cB_{\kappa}G$; by definition of $\univ X$, we obtain a morphism $\univ X \arr \cB_{\kappa}G$, together with an isomorphism of the image of the object in $\univ X(\kappa)$ corresponding to the composite $\spec \kappa \xarr{x_{0}} X \arr \univ X$ with the trivial torsor in $\cB_{\kappa}G(\kappa)$. By the discussion above, this yields a homomorphism of group schemes $\pi(X, x_{0}) \arr G$.

This gives a bijective correspondence between isomorphism classes of $G$-torsors $P \arr X$ with a trivialization $x_{0}^{*}P \simeq G$, and homomorphism of group schemes $\pi(X, x_{0}) \arr G$. Thus, in the case of stacks with a given rational point, the Nori fundamental gerbe corresponds to the Nori fundamental group.
\end{remark}

\section{Base change for the Nori fundamental gerbe}\label{sec:base-change}

Nori showed in \cite[II, Proposition~5]{nori-phd} that the formation of the fundamental group scheme satisfies base change for algebraic separable extensions. Here we prove the analogous result for fundamental gerbes under finite separable extensions, which is very useful for applications and calculations (see Section~\ref{sec:examples}). Nori used the action of the Galois group on the fundamental group scheme; this can be made to work in our case too, using the theory of group actions on stacks (see \cite{romagny-group-actions}); however, this is technically rather involved, so we prefer a different method, based on the Weil restriction of algebraic stacks.

\begin{proposition}\label{prop:base-change}
Let $\kappa'/\kappa$ be a separable extension, $X$ an inflexible fibered category over $\aff \kappa$. Suppose that either

\begin{enumeratea}

\item $\kappa'$ is finite over $\kappa$, or

\item there exists a quasi-compact scheme $U$ and a morphism $U \arr X$ which is representable, faithfully flat, quasi-compact and quasi-separated.

\end{enumeratea}

 Then $X_{\kappa'}$ is inflexible over $\kappa'$, and $\univ[\kappa']{X_{\kappa'}} = \spec\kappa' \times \univ X$.
\end{proposition}

Let us describe the essential features of the Weil restriction of stacks. If $A$ is a finite $\kappa$-algebra and $Y \arr \aff A$ is a fibered category, we define the Weil restriction $\rR_{A/\kappa}Y$, as usual, as the fibered product $\aff \kappa \times_{\aff A} Y$, where the functor $\aff \kappa \arr \aff A$ is defined by $S \arrto S_{A} \eqdef \spec A\times_{\spec \kappa}S$. As a pseudo-functor, $\rR_{A/\kappa}Y$ is defined by $S \arr Y(S_{A})$. When $Y$ is represented by a scheme, then $\rR_{A/\kappa}Y$ is represented by its Weil restriction, which is a scheme when $A$ is finite over $\kappa$. Furthermore, it is immediate to check that if $Y$ is a stack, say in the étale topology, so is $\rR_{A/\kappa}Y$.

This construction is clearly functorial, that is, every morphism $f\colon Y \arr Y'$ of fibered categories over $\aff A$ induces a morphism $\rR_{A/\kappa}f\colon \rR_{A/\kappa}Y \arr \rR_{A/\kappa}Y'$.

When $X \arr \aff \kappa$ is fibered category, we have a natural equivalence of categories of base-preserving functors
   \[
   \hom_{\kappa}(X, \rR_{A/\kappa}Y) \simeq \hom_{A}(X_{A}, Y)\,,
   \]
where $X_{A}$ is the fibered product $\aff A \times_{\aff \kappa}X$ (i.e., the Weil restriction is left adjoint to the pullback functor of fibered categories, in the 2-categorical sense). Because of this, one sees immediately that Weil restriction preserves fibered products, that is, if $Y' \arr Y$ and $Y'' \arr Y$ are morphisms of fibered categories over $\aff A$, the natural morphism $\rR_{A/\kappa}(Y' \times_{Y}Y'') \arr \rR_{A/\kappa}Y' \times_{\rR_{A/\kappa}Y} \rR_{A/\kappa}Y''$ is an equivalence.

If $\ell$ is an extension of $\kappa$, and $A_{\ell} \eqdef \ell\otimes_{\kappa}A$, it is easy to see that $(\rR_{A/\kappa}Y)_{\ell} \simeq \rR_{A_{\ell}/\ell}(Y_{A_{\ell}})$.

\begin{lemma}\call{properties-weil-restriction}\hfil
Suppose that $\kappa'$ is a finite separable extension of $\kappa$, and $\Gamma$ is a finite stack over $\kappa'$. Then $\rR_{\kappa'/\kappa}\Gamma$ is finite stack over $\kappa$.

\end{lemma}

\begin{proof}
Call $n$ the degree of $\kappa'$ over $\kappa$. Let $\kappa\sep$ be the separable closure of $\kappa$, and call $\nu_{1}$, \dots,~$\nu_{n}$ the embeddings of $\kappa'$ into $\kappa\sep$. Set $A \eqdef \kappa\sep \otimes_{\kappa}\kappa'$; we have an isomorphism of $\kappa$-algebras $A \simeq (\kappa\sep)^{n}$, such that for each $i \in I$ the composite of the embedding $\kappa' \subseteq A$ with the $i\th$ projection is $\nu_{i}$. For each stack $Y$ over $\kappa'$, denote by $Y_{i}$ the fibered product $\spec\kappa\sep\times_{\spec \kappa'} Y$, where the morphism $\spec \kappa\sep \arr \spec \kappa'$ is induced by $\nu_{i}\colon \kappa' \arr \kappa\sep$.

If $Y$ is a stack over $\kappa'$, then we have $Y_{A} = \coprod_{i=1}^{n}Y_{i}$; hence it is easy to see that
   \[
   (\rR_{\kappa'/\kappa}Y)_{\kappa\sep} = \rR_{A/\kappa\sep}\biggl(\,\coprod_{i=1}^{n}Y_{i}\biggr)
   = \prod_{i=1}^{n}Y_{i}\,;
   \]
hence if $Y$ is a finite scheme over $\kappa$, then $\rR_{\kappa'/\kappa}Y$ is a finite scheme over $\kappa$.

In the general case, take an fppf presentation $R \double U$ of $\Gamma$, where $U$ and $R$ are finite schemes over $\kappa'$. Then $\prod_{i}R_{i} \double \prod_{i}U_{i}$ is an fppf presentation of $\prod_{i}\Gamma_{i}$, hence $\rR_{\kappa'/\kappa}R \double \rR_{\kappa'/\kappa}U$ is an fppf presentation of $\rR_{\kappa'/\kappa}\Gamma$. Since $\rR_{\kappa'/\kappa}R$ and $\rR_{\kappa'/\kappa}U$ are finite schemes, the proof is complete.
\end{proof}

\begin{proof}[Proof of Proposition~\ref{prop:base-change}]

Let us show that $(\univ X)_{\kappa'}$ is a universal gerbe for $X_{\kappa'}$. For this, first assume that $\kappa'$ is finite over $\kappa$.

Let $\Gamma$ be a finite stack over $\kappa'$, and consider the functor
   \[
   \hom_{\kappa'}\bigl((\univ X)_{\kappa'}, \Gamma\bigr) \arr
   \hom_{\kappa'}\bigl(X_{\kappa'}, \Gamma\bigr)
   \]
induced by $X_{\kappa'} \arr (\univ X)_{\kappa'}$. Under the adjunction between Weil restrictions and pullbacks, we have equivalences
   \[
   \hom_{\kappa'}\bigl((\univ X)_{\kappa'}, \Gamma\bigr) \simeq 
   \hom_{\kappa}\bigl(\univ X, \rR_{\kappa'/\kappa}\Gamma\bigr)
   \]
and
   \[
   \hom_{\kappa'}\bigl(X_{\kappa'}, \Gamma\bigr) \simeq 
   \hom_{\kappa}\bigl(X, \rR_{\kappa'/\kappa}\Gamma\bigr)\,.
   \]
However, by the Lemma above $\rR_{\kappa'/\kappa}\Gamma$ is a finite stack, so the natural functor 
   \[
   \hom_{\kappa}\bigl(\univ X, \rR_{\kappa'/\kappa}\Gamma\bigr) \arr
   \hom_{\kappa}\bigl(X, \rR_{\kappa'/\kappa}\Gamma\bigr)
   \]
is an equivalence. This completes the proof in the case that $\kappa'$ is finite over $\kappa$.

Assume that we are under the hypothesis~(b): choose a morphism $U \arr X$ as in condition~(b), and set $R \eqdef U \times_{X} U$. Suppose that we are given a finite stack $\Gamma$ over $\kappa'$; we need to show that the functor
   \[
   \hom_{\kappa'}\bigl((\univ X)_{\kappa'}, \Gamma\bigr)
   \arr \hom_{\kappa'}\bigl(X_{\kappa'}, \Gamma\bigr)
   \]
is an equivalence. Let us show that the functor is essentially surjective; for this, choose a morphism $X_{\kappa'} \arr \Gamma$. By descent theory, this corresponds to an object $\xi$ of $\Gamma(U_{\kappa'})$, with descent data, that is, with an isomorphism of the two pullback of $\xi$ to $\Gamma(R_{\kappa'})$, satisfying the cocycle condition. Since the stack $\Gamma$ is finitely presented, there exists an intermediate extension $\kappa \subseteq \ell \subseteq \kappa'$ finite over $\kappa$, and a finite stack $\Delta$ on $\ell$, such that $\Gamma \simeq \Delta_{\kappa'}$. Furthermore, by enlarging $\ell$ we may assume that there exists an object $\eta$ of $\Delta(U_{\ell})$ with descent data in $\Delta(R_{\ell})$, whose pullback to $\Gamma(U_{\kappa'})$ is isomorphic to $\xi$, as an object with descent data. This gives a morphism $X_{\ell} \arr \Delta$ which pulls back to the given morphism $X_{\kappa'} \arr \Gamma$. Since $\ell$ is finite over $\kappa$ we have that $X_{\ell} \arr \Delta$ factors through $(\univ X)_{\ell}$; hence $X_{\kappa'} \arr \Gamma$ factors through $\bigl((\univ X)_{\ell}\bigr)_{\kappa'} = (\univ X)_{\kappa'}$, as claimed.

The proof that the functor is fully faithful is similar, and left to the reader.
\end{proof}

\begin{remark}
The base-change result fails for inseparable extensions; see \cite[p. 89]{nori-phd}, \cite{mehta-subramanian-fundamental-group-scheme} and \cite{pauly-nori-counterexample}.
\end{remark}

\section{The tannakian interpretation of the fundamental gerbe}\label{sec:tannaka}

In this section all schemes, and all algebraic spaces, will be quasi-separated. A morphism of fibered categories is called \emph{representable} when it is represented by algebraic spaces.

\begin{definition}\call{def:psp}
A fibered category $X$ over $\kappa$ is \emph{\psp} if it satisfies the following two conditions.

\begin{enumeratea}

\itemref{a} There exists a quasi-compact scheme $U$ and a morphism $U \arr X$ which is representable, faithfully flat, quasi-compact and quasi-separated.

\itemref{b} For any locally free sheaf of $\cO_{X}$-modules $E$ on $X$, the $\kappa$-vector space $\H^{0}(X, E)$ is finite-dimensional.

\end{enumeratea}

\end{definition}

Notice that in the definition above we don't assume that $X$ is a stack, in any topology. Given a morphism $U \arr X$ as in part~\refpart{def:psp}{a}, we obtain an fpqc groupoid $R \double U$, where $R \eqdef U \times_{X} U$. If $X$ is a stack in the fpqc topology, then it is equivalent to the quotient stack of $(R \double U)$-torsors.

\begin{examples}\hfil
\label{expl:pseudo-propre}
	
\begin{enumeratea}

\item A finite algebraic stack is \psp. 

\item An affine fpqc gerbe is \psp.

\end{enumeratea}

This is clear for a finite stack.

If $X$ is an affine fpqc gerbe, condition \refpart{def:psp}{a} of Definition~\ref{def:psp} is obviously verified. To prove condition \refpart{def:psp}{b}, let $E$ be a locally free sheaf on $X$. Because of Remark~\ref{rmk:base-change} below we can make a base extension, and assume that $X(\kappa) \neq \emptyset$. Then $X = \cB_{\kappa}G$, where $G$ is an affine group scheme; but then a locally free sheaf on $X$ is a finite-dimensional representation of $G$, $\H^{0}(X, E) = E^{G}$, and the result is obvious.
\end{examples}

\begin{remark}\label{rmk:base-change}
We will use the following fact. Suppose that $X$ is a fibered category, and let $U \arr X$ be a morphism as in Definition \refall{def:psp}{a}; set $R \eqdef U \times_{X} U$. If $E$ is a locally free sheaf on $X$ we denote by $E_{U}$ and $E_{R}$ the restrictions of $E$ to $U$ and $R$ respectively. Then $\H^{0}(X, E)$ is the equalizer of the pullbacks $\H^{0}(U, E_{U}) \double \H^{0}(R, E_{R})$ (this is a straightforward application of descent theory).

Furthermore, in the situation above, let $\kappa'$ be a field extension. Then the induced morphism $U_{\kappa'} \arr X_{\kappa'}$ is also representable, faithfully flat, quasi-compact and quasi-separated, and $R_{\kappa'} = U_{\kappa'}\times_{X_{\kappa'}} U_{\kappa'}$. Since $U$ is quasi-compact and quasi-separated we have $\H^{0}(U, E_{U}) \otimes_{\kappa}\kappa' = \H^{0}(U_{\kappa'}, E_{U_{\kappa'}})$, and analogously for $R$. Hence we also have $\H^{0}(X, E) \otimes_{\kappa} \kappa' = \H^{0}(X_{\kappa'}, E_{\kappa'})$.

\end{remark}

\begin{lemma}
If $X$ is \infl and \psp, then $\H^{0}(X, \cO) = \kappa$.
\end{lemma}

\begin{proof}
Since $X$ is \psp, the $\kappa$-vector space $\H^{0}(X, \cO)$ is finite-dimensional. The result follows from Proposition \refall{prop:properties-infl}{a}.
\end{proof}

Let $\cA$ be a $\kappa$-linear rigid tensor category, with finite-dimensional Hom vector spaces, in which the idempotents split. (For example, if $X$ is a \psp fibered category we can take $\cA = \vect X$, see \cite{atiyah-krull-schmidt}, Lemma 6). We can define the \emph{indecomposable K-theory ring $\ikt \cA$} as the Grothendieck group associated with the monoid of isomorphism classes of objects of $\cA$, with the product given by tensor product. The Krull--Schmidt theorem holds in the category $\cA$ (see \cite{ringel_tame_1984}, \S 2.2); hence $\ikt \cA$ is a free abelian group on the isomorphism classes of indecomposable objects of $\cA$, and two objects of $\cA$ are isomorphic if and only if their classes in $\ikt \cA$ coincide.

Notice that if $f \in \NN[x]$ is a polynomial with natural numbers as coefficient and $E$ is an object of $\cA$, we can define $f(E)$, by interpreting the sum as a direct sum, and a power as a tensor power. 

\begin{definition}
We say that an object $E$ of $\cA$ is \emph{finite} when one of the following equivalent conditions is satisfied.

\begin{enumeratea}

\item The class of $E$ in $\ikt \cA$ is integral over $\ZZ$.

\item The class of $E$ in $\ikt \cA \otimes_{\ZZ}  \QQ$ is algebraic over $\QQ$.

\item There exist $f$ and $g$ in $\NN[x]$ with $f \neq g$ and $f(E) \simeq g(E)$.

\item The set of isomorphism classes of indecomposable components of all the powers of $E$ is finite.

\end{enumeratea}
\end{definition}

The equivalence of these conditions is proved as in \cite[2.3]{nori-phd}.

\begin{proposition}[\hbox{\cite[Lemma~(3.1)]{nori-phd}}] \label{prop:properties-finite}
\hfil
\begin{enumeratea}

\item Finite sums, tensor products and duals of finite objects are finite.

\item If the direct sum of two objects is finite, both objects are finite.

\end{enumeratea}
\end{proposition}

Our definition of essentially finite sheaf is more elementary and direct than Nori's, and works more generally.

\begin{definition}
An object $E$ of $\cA$ is \emph{essentially finite} if it is the kernel of a homomorphism between two finite objects.
\end{definition}

If $X$ is a \psp fibered category and $E$ is a locally free sheaf on $X$, then we say that $E$ is finite, or essentially finite, when it has the corresponding property when viewed as an object of $\vect X$.

\begin{proposition}\label{prop:all-essentially-finite}
Let $\Phi$ be a profinite gerbe over $\kappa$. Then all representations of $\Phi$ are essentially finite.

Furthermore, if the characteristic of $\kappa$ is $0$ all representations are finite, and the category $\rep \Phi$ is semisimple.
\end{proposition}

\begin{proof}
Since $\rep \Phi$ is the colimit of categories $\rep \Gamma$, where $\Gamma$ is finite, it is enough to show both parts when $\Phi = \Gamma$ is finite. Then the first part follows from Lemma \refall{lem:reduced-finite-stacks}{a}.

For the second part it is enough to check that the functor $\H^{0}\colon \rep \Gamma \arr \vect_{\kappa}$ is exact. For this we can make a finite extension of the base field, and assume that there exists a section $\spec\kappa \arr \Gamma$. Then $\Gamma$ is of the type $\cB_{\kappa}G$, where $G$ is a finite group scheme on $\kappa$, and the result is standard.
\end{proof}

\begin{theorem}\label{thm:tannakian-description}
Suppose that $X$ is \infl and \psp  over $\kappa$. Then the pullback functor $\rep \univ X \arr \vect X$ gives an equivalence of tensor categories of $\rep \univ X$ with $\efin X$.

Furthermore, if the characteristic of $\kappa$ is $0$ we have $\fin X = \efin X$.
\end{theorem}

Here is a purely tannakian consequence of the Theorem.

\begin{corollary}
Let $\cA$ be a Tannaka category. The full subcategory of $\cA$ consisting of essentially finite objects is tannakian. Furthermore, if $\cha \kappa = 0$, then every essentially finite object of $\cA$ is finite.
\end{corollary}

We don't know a purely tannakian proof of this, that does not use the formalism of affine gerbes.

\begin{proof}
	Let $\Phi$ be an affine fpqc gerbe, such that $\cA$ is equivalent to $\rep \Phi$. Then $\Phi$ is \infl (Proposition \refall{prop:properties-infl}{b}) and pseudo-proper (Examples \ref{expl:pseudo-propre}). By Theorem~\ref{thm:tannakian-description}, the category of essentially finite objects in $\cA$ is equivalent to the category of representations of the universal gerbe $\univ{\Phi}$, and the result follows.
\end{proof}

For the proof of the Theorem~\ref{thm:tannakian-description}, call $F\colon X \arr \univ X$ the morphism, and consider $F^{*}\colon \rep \univ X \arr \vect X$. The functor $F^{*}$ is an exact functor of tensor categories; as such, it carries finite representations into finite locally free sheaves, and essentially finite representations into essentially finite locally free sheaves. By Proposition~\ref{prop:all-essentially-finite} the image of $F^{*}$ is contained in $\efin X$.

Let us check that $F^{*}$ is fully faithful. The category $\rep \univ X$ is the colimit of the categories $\rep \Gamma$ over the category of Nori-reduced morphism $X \arr \Gamma$; hence it is enough to show that if $f\colon X \arr \Gamma$ is a Nori-reduced morphism, the pullback functor $f^{*}\colon \rep \Gamma \arr \vect X$ is fully faithful; this follows from Lemma~\ref{lem:pullback-faithfully-flat} and from the following.

\begin{lemma}\label{lem:Nori-reduced-coh-flat}
Let $f\colon X \arr \Gamma$ be a Nori-reduced morphism. Then $f_{*}\cO_{X} = \cO_{\Gamma}$.
\end{lemma}

\begin{proof}
From Lemma~\ref{lem:image-coherent} below, we have that $f_{*}\cO_{X}$ is a coherent sheaf of $\cO_{\Gamma}$-algebras; let $\Gamma' \arr \Gamma$ be its relative spectrum. Then $\Gamma'$ is a finite stack, the morphism $\Gamma' \arr \Gamma$ is finite and representable, and $f$ factors as $X \xarr{f'}\Gamma' \arr \Gamma$. Since $X$ is \infl, there exists a closed subgerbe $\Gamma'' \subseteq \Gamma'$ such that $f'$ factors through $\Gamma''$. If $I \subseteq \cO_{\Gamma'}$ is the sheaf of ideals of $\Gamma''$ in $\Gamma'$, we have that all the elements of $I$ pull back to zero on $X$. By the definition of $\Gamma'$ it follows that $I = 0$, so $\Gamma' = \Gamma''$ is a gerbe. But $f$ is Nori-reduced, so $\Gamma' = \Gamma$, and $f_{*}\cO_{X} = \cO_{\Gamma}$.
\end{proof}

Next we need to show that every essentially finite locally free sheaf is isomorphic to a pullback from $\univ X$. Since the functor $F^{*}$ is exact and the category $\rep \univ X$ is abelian, we may assume that $E$ is finite.

Let $f$ and $g$ be distinct polynomials in $\NN[x]$, such that there exists an isomorphism $\sigma$ of $f(E)$ with $g(E)$. Let $r$ be the rank of $E$. Set $V \eqdef k^{r}$, and denote by $I$ the scheme representing the isomorphisms of $f(V)$ with $g(V)$. It is isomorphic to $\GL_{N}$, where $N \eqdef f(r) = g(r)$; in particular, it is affine. There is a natural left action of $\GL_{r}$ on $I$; the isomorphism $\sigma$ gives a lifting $X \arr [I/\GL_{r}]$ of the morphism $X \arr \cB_{\kappa}\GL_{r}$ corresponding to the locally free sheaf $E$. We need to show that the morphism $X \arr \cB_{\kappa}\GL_{r}$ factors through a finite gerbe; for this it is sufficient to prove that the scheme-theoretic image of $X$ in $[I/\GL_{r}]$ is a finite stack.

\begin{lemma}\label{lem:finite-stabilizers}
The action of $\GL_{r}$ on $I$ has finite stabilizers.
\end{lemma}

Let us take this for granted for the time being, and let us complete the proof. By the Lemma, all geometric orbits of $\GL_{r}$ on $I$ are closed. Since $I$ is affine and $\GL_{r}$ is geometrically reductive, there exists an affine geometric quotient $I \arr I/\GL_{r}$, whose geometric fibers are, set-theoretically, the geometric orbits of the action of $\GL_{r}$ on $I$. The composite $X \arr [I/\GL_{r}] \arr I/\GL_{r}$ must factor through a rational point $\spec\kappa \arr I/\GL_{r}$, since $\H^{0}(X, \cO_{X}) = \kappa$ and $I/\GL_{r}$ is affine. Call $\Omega$ the fiber of $I$ over this rational point; the morphism $X \arr [I/\GL_{r}]$ factors through $[\Omega/\GL_{r}]$. To conclude it is enough to show that $[\Omega/\GL_{r}]$ is a finite stack. But $[\Omega/\GL_{r}](\overline{\kappa})$ is a connected groupoid, and has quasi-finite diagonal, since the stabilizers are finite, so the result follows from Proposition~\ref{prop:char-finite}.

\begin{proof}[Proof of Lemma~\ref{lem:finite-stabilizers}]

To prove the Lemma we may extend the base field $\kappa$, and assume that it is algebraically closed. Let $\phi\colon f(V) \simeq g(V)$ be an isomorphism, and call $G$ its stabilizer; we need to show that $G$ is finite. Since the Krull--Schmidt property holds, we may assume that $\deg f \neq \deg g$.

Let $H$ be a subgroup of $G$; then $V$ has the property that $f(V)$ is isomorphic to $g(V)$ as a representation of $H$.

If $G$ is positive-dimensional, then it must contain either a copy of $\ga$ or of $\gm$; hence it is enough to show that if $H = \ga$ or $H = \gm$ and $V$ is a faithful representation of $H$, then $f(V) \not\simeq g(V)$.

For $H = \gm$ this is easy; since every representation of $\gm$ is semisimple, two representations of $\gm$ that have the same class in the ring of representations of $\gm$ are isomorphic. The ring of representations of $\gm$ is well known to be isomorphic to $\ZZ[t^{\pm 1}]$; since $\ZZ$ is algebraically closed in $\ZZ[t^{\pm 1}]$ and the class of $V$ is not in $\ZZ$, because $V$ is not trivial, $f(V) \not\simeq g(V)$.

Suppose that $H = \ga$. For any representation $V$ of $\ga$, define the $\delta$-invariant $\delta(V)$ as follows. Fix a basis of $V$, giving an isomorphism $\GL(V) \simeq \GL_{n}$. Write the action as an invertible matrix whose entries are polynomials in $\kappa[t]$; the largest degree of one of these polynomials is $\delta(V)$. It is immediate to see that $\delta(V)$ does not depend on the basis. Furthermore, we have that $\delta(V \oplus W) = \max\bigl(\delta(V), \delta(W)\bigr)$, $\delta(V \otimes W) = \delta(V) + \delta(W)$, and $\delta(V) = 0$ if and only if $V$ is trivial. Hence we have $\delta\bigl(f(V)\bigr) = (\deg f)\delta(V)$ and $\delta\bigl(g(V)\bigr) = (\deg g)\delta(V)$. Since $\delta(V) > 0$ and $\deg f \neq \deg g$ we have $\delta\bigl(f(V)\bigr) \neq \delta\bigl(g(V)\bigr)$, and $f(V) \not\simeq g(V)$.
\end{proof}

This completes the proof of Theorem~\ref{thm:tannakian-description}.

It is interesting to observe this theorem has a converse, showing that indeed the concept of \infl stack is a very natural one.

\begin{theorem}\label{thm:char-infl}
Let $X$ be a \psp fibered category over $\spec \kappa$. Then $X$ is \infl if and only if the tensor category $\efin X$ is tannakian.
\end{theorem}

\begin{proof}
We have already seen that if $X$ is inflexible, then $\efin X$ is tannakian. So, suppose that $\efin X$ is tannakian.

Let $\Gamma$ be a finite stack with a morphism $X \arr \Gamma$; we need to show that the stack-theoretic image $\Gamma'$ of $X$ in $\Gamma$ is a gerbe. By Lemma~\ref{lem:image-coherent}, the sheaf $f_{*}\cO_{X}$ is coherent. If $\Gamma'$ is its relative spectrum over $\Gamma$, the morphism $f\colon X \arr \Gamma$ factors through $\Gamma'$; by Proposition \refall{prop:properties-infl}{b}, it is enough to prove that $\Gamma'$ is a gerbe. By replacing $\Gamma$ with $\Gamma'$, we may assume that $\cO_{\Gamma} = f_{*}\cO_{X}$.

\begin{lemma}\label{lem:exact-subcategory}
If the category $\efin X$ is tannakian, it is an exact abelian subcategory of the category of sheaves of $\cO_{X}$-modules.
\end{lemma}

\begin{proof}
Let $\Phi$ be an fpqc gerbe with an equivalence of Tannaka categories $\rep \Phi \simeq \efin X$. This equivalence is realized by a morphism $f\colon X \arr \Phi$. The pullback from $\rep \Phi$ to sheaves of $\cO_{X}$-modules is exact, by Proposition \refall{prop:properties-fpqc-gerbe}{b}, and this proves the Lemma.
\end{proof}

We have $\H^{0}(X, \cO) = \kappa$, because $\cO_{X}$ is the unit in the tannakian category $\efin X$. Since $\H^{0}(\Gamma, \cO) = \H^{0}(X, \cO) = \kappa$, we see that $\Gamma$ is geometrically connected. Now let us show that $\Gamma$ is reduced. Let $N \subseteq \cO_{\Gamma}$ the sheaf of nilpotent sections. Let $\rho\colon U \arr \Gamma$ a faithfully flat morphism from a finite connected scheme $U$, and call $U_{0}$ the inverse image of $\Gamma_{\mathrm{red}}$ in $U$; this is the subscheme whose sheaf of ideals is $\rho^{*}N$. Choose a surjective homomorphism $\cO_{U}^{n} \arr \rho^{*}N$; applying $\rho_{*}$, which is exact, because $\rho$ is representable and finite, and then pulling back along $f$ we obtain an exact sequence
   \[
   f^{*}\rho_{*}\cO_{U}^{n} \arr f^{*}\rho_{*}\cO_{U} \arr f^{*}\rho_{*}\cO_{U_{0}} \arr 0\,.
   \]
From Lemmas~\ref{lem:exact-subcategory} and \refall{lem:reduced-finite-stacks}{c},  we have that $f^{*}\rho_{*}\cO_{U_{0}}$ is locally free. Restricting to a point of $X$ we see that it can not be $0$; hence its annihilator in $\cO_{X}$ must be $0$. Since $\cO_{\Gamma} = f_{*}\cO_{X}$, the annihilator of $\rho_{*}\cO_{U_{0}}$ in $\cO_{\Gamma}$, which is $N$, must also be $0$. So $\Gamma$ is reduced, as claimed. If $\kappa$ is perfect, this is enough to conclude, by Lemma~\ref{prop:gerbe<->geom}. When $\kappa$ is not perfect we need some additional work to show that $\Gamma$ is geometrically reduced.

From Lemma~\ref{lem:reduced-finite-stacks}, we have the equality $\efin \Gamma = \coh \Gamma$. From the equality $\H^{0}(\Gamma, \cO) = \kappa$ and from the fact that $\efin \Gamma$ is an abelian category, we see that $\efin \Gamma$ is tannakian. Let $\Phi$ be the corresponding fpqc gerbe; the equivalence $\coh \Phi = \rep \Phi \simeq \efin \Gamma$ is realized as pullback along a morphism $\phi\colon \Gamma\arr \Phi$. 

Let $\kappa'$ be a finite extension of $\kappa$. The category of coherent sheaves on $\Gamma_{\kappa'}$ is equivalent to the category of coherent sheaves $F$ on $\Gamma$, with a homomorphism of $\kappa$-algebras $\kappa' \arr \End_{\cO_{\Gamma}}(F)$, and analogously for $\Phi$. Hence pullback along the natural morphism $\Gamma_{\kappa'} \arr \Phi_{\kappa'}$ yields an equivalence of categories between $\coh \Gamma_{\kappa'}$ and $\coh \Phi_{\kappa'}$. Since $\Phi_\kappa'$ is a gerbe we have $\coh \Phi_{\kappa'} = \vect \Phi_{\kappa'}$; it follows that every coherent sheaf on $\Gamma_{\kappa'}$ is locally free. This implies that $\Gamma_{\kappa'}$ is reduced (for otherwise the structure sheaf of $(\Gamma_{\kappa'})_{\mathrm{red}}$ would not be locally free). This shows that $\Gamma$ is geometrically reduced, and completes the proof of the Theorem.
\end{proof}

\subsection*{Some technical lemmas} Here we collect some lemmas that were used in the proof of the results above, in order to unclutter the exposition.

\begin{lemma}\call{lem:reduced-finite-stacks}
Let $\Gamma$ be a finite stack over $\kappa$.

\begin{enumeratea}

\itemref{c} Let $\rho\colon T \arr \Gamma$ be a faithfully flat morphism, where $T$ is a connected finite scheme over $\kappa$. Then every locally free sheaf on $\Gamma$ is a subsheaf of $(\rho_{*}\cO_{T})^{\oplus r}$ for some $r \geq 0$, and $\rho_{*}\cO_{T}$ is a finite locally free sheaf.

\itemref{a} Assume moreover that $\Gamma$ is reduced. Then every coherent sheaf on $\Gamma$ is locally free, and essentially finite.

\itemref{b} Assume again that $\Gamma$ is reduced. Every morphism from an algebraic stack to $\Gamma$ is flat.

\end{enumeratea}
\end{lemma}

\begin{proof}For all this, we may assume that $\Gamma$ is connected.

For \refpart{lem:reduced-finite-stacks}{c}, call $d$ the degree of $\rho$. Since $F$ is locally free, the sheaf $\rho^{*}F$ is free over $T$ (since $T$ is the spectrum of a local artinian ring, every locally free sheaf on $T$ is free); fix an isomorphism $\rho^{*}F \simeq \cO_{T}^{\oplus r}$. The adjunction homomorphism $F \arr \rho_{*}\rho^{*}F \simeq (\rho_{*}\cO_{T})^{\oplus r}$ is injective. This proves the first part of the statement. For the second part, it follows from the projection formula that
 
\begin{align*}
   \rho_{*}\cO_{T} \otimes \rho_{*}\cO_{T} &\simeq
   \rho_{*}(\cO_{T} \otimes \rho^{*}\rho_{*}\cO_{T})\\
   & \simeq \rho_{*}(\cO_{T} \otimes \cO_{T}^{\oplus d})\\
   & \simeq \rho_{*}(\cO_{T})^{\oplus d}\,.\\
   \end{align*}

For \refpart{lem:reduced-finite-stacks}{a}, take a smooth surjective morphism $\pi\colon U \arr \Gamma$, where $U$ is a scheme. Let $F$ be a coherent sheaf on $\Gamma$. Then $\pi^{*}F$ is a coherent sheaf on the reduced noetherian scheme $U$; such a sheaf is locally free on an open dense subscheme $V \subseteq U$. But $\Gamma$ is finite, so $V$ also surjects onto $\Gamma$, hence $F$ is locally free. Now take a faithfully flat morphism $\rho\colon T \arr \Gamma$, where $T$ is a connected finite scheme over $\kappa$. By embedding the cokernel of $F \arr \rho_{*}\rho^{*} F$ into a finite representation, we see that $F$ is essentially finite.

As to \refpart{lem:reduced-finite-stacks}{b}, notice that it is enough to prove that every morphism from an affine scheme to $\Gamma$ is flat. The diagonal of $\Gamma$ is affine, so such a morphism is affine, and it is enough to show that all quasi-coherent sheaves on $\Gamma$ are flat. By \cite[Proposition 15.4]{laumon-moret-bailly} every quasi-coherent sheaf on $\Gamma$ is a colimit of coherent sheaves, so the statement follows from \refpart{lem:reduced-finite-stacks}{a}.
\end{proof}

\begin{lemma}\label{lem:image-coherent}
Suppose that $f\colon X \arr \Gamma$ is a morphism of $\kappa$-algebraic stacks, where $X$ is \psp and $\Gamma$ is finite. Then $f_{*}\cO_{X}$ is a coherent sheaf of $\cO_{\Gamma}$-modules.
\end{lemma}

\begin{proof}
First of all, let us show that $f_{*}\cO_{X}$ is quasi-coherent. Choose a faithfully flat quasi-compact morphism $U \arr X$, and set $R \eqdef U \times_{X} U$. Call $f_{U}\colon U \arr \Gamma$ and $f_{R}\colon R \arr \Gamma$ the composites of $f$ with the morphisms $U \arr X$ and $R \arr X$. Then it is easily checked that $f_{*}\cO_{X}$ is the equalizer of the two morphisms ${f_{U}}_{*}\cO_{U} \double {f_{R}}_{*}\cO_{R}$ resulting from the projections $R \double U$. Since $U$ and $R$ are quasi-compact and quasi-separated, ${f_{U}}_{*}\cO_{U}$ and ${f_{R}}_{*}\cO_{R}$ are quasi-coherent, so the result follows.

To prove that $f_{*}\cO_{X}$ is coherent, choose a faithfully flat morphism $\pi\colon T \arr \Gamma$, where $T$ is a finite $\kappa$-scheme, and consider the cartesian diagram
   \[
   \xymatrix{
   X' \ar[r] ^-{f'} \ar[d]^{\rho} &T \ar[d]^\pi\\
   X  \ar[r] ^-{f}                &\Gamma\,.
   }
   \]
Clearly, $X'$ is \psp.
Since $\cO_{X}$ is contained in $\rho_{*}\cO_{X'}$, because $\rho$ is faithfully flat, it is enough to show that $f_{*}\rho_{*}\cO_{X'} = \pi_{*}f'_{*}\cO_{X'}$ is coherent. But $f'_{*}\cO_{X'}$ is a coherent sheaf of $\cO_{T}$-algebras, because $\H^{0}(X', \cO)$ is a finite dimensional vector space over $\kappa$. Since $\pi$ is finite, we have that $\pi_{*}f'_{*}\cO_{X'}$ is coherent.
\end{proof}

\begin{lemma}\label{lem:pullback-faithfully-flat}
Suppose that $f\colon X \arr Y$ is a morphism of fibered categories. Suppose that the natural homomorphism $\cO_{Y} \arr f_{*}\cO_{X}$ is an isomorphism. Then the pullback functor $f^{*}\colon \vect Y \arr \vect X$ is fully faithful.
\end{lemma}

\begin{proof}
This is a standard application of the projection formula.
\end{proof}

\section{The étale fundamental gerbe}\label{sec:etale}

Let $X$ be an \infl algebraic stack, $\univ X = \projlim_{i} \Gamma_{i}$ its fundamental gerbe, expressed as a projective limit along the category $I$ whose objects are Nori-reduced morphisms $X \arr \Gamma_{i}$. Consider the full subcategory $I^{\et}$ of $I$ whose objects consist of Nori-reduced morphisms $X \arr \Gamma_{j}$ in which $\Gamma_{j}$ is étale. In characteristic~$0$ we have $I^{\et} = I$. Since the fibered product of two étale gerbes is an étale stack, we have  that $I^{\et}$ is a 2-cofiltered category. 

\begin{definition}
The étale fundamental gerbe of $X$ is the profinite gerbe $\etuniv X \eqdef \projlim_{i \in I^{\et}} \Gamma_{i}$.
\end{definition}

Since $I^{\et}$ is a subcategory of $I$, there is a natural projection $\univ X  \arr \etuniv X$. The natural morphism $X \arr \etuniv X$ is easily seen to be universal among all morphisms to a pro-étale gerbe.

The profinite étale gerbe $\etuniv X$ is isomorphic to the gerbe associated with the relative fundamental groupoid introduced by Deligne in \cite[10.17-18]{deligne-minus-three-points}.

First of all, let us recall Deligne's construction. Suppose that $X$ is a geometrically connected algebraic stack over $\kappa$. Let $Y \arr X$ be a connected Galois cover\footnote{By \emph{cover} we mean a representable morphism that is finite and étale.}. This corresponds to a morphism $X \arr \cB_{\kappa}(\aut Y)$, where $\aut Y \eqdef \aut_{X}Y$ is the Galois group, which is an fpqc cover; hence $\cB_{\kappa}(\aut Y)$ is equivalent to the stack of fpqc torsors under the groupoid $X \times_{\cB_{\kappa}(\aut Y)} X \double X$. The fibered product $X \times_{\cB_{\kappa}(\aut Y)} X$ is equivalent to the $X$-scheme $P_{Y} \eqdef \underisom_{X\times_S X}^{\aut Y}(\pr_{2}^{*}Y, \pr_{1}^{*}Y)$ of $\aut Y$-equivariant isomorphisms of the two pullbacks of $Y$. Let $P_{Y}^{0}$ be the connected component of $P_{Y}$ containing the image of the diagonal $X \arr X \times_{\cB_{\kappa}(\aut Y)} X$; then $P^{0}_{Y} \double X$ is an fpqc groupoid.

Now, let $Z \arr X$ be another connected Galois étale cover of $X$, and let $f\colon Z \arr Y$ be a morphism of coverings. Then $f$ induces a morphism $P_{Z} \arr P_{Y}$, which is easily seen to be independent of $f$; this sends $P^{0}_{Z}$ into $P^{0}_{Y}$, giving a morphism of groupoids from $P^{0}_{Z}\double X$ to $P^{0}_{Y}\double X$.

Let $\{Y_{j}\}_{j \in J}$ be a set of representatives for isomorphism classes of Galois connected covers of $X$; we introduce a partial ordering on $J$, saying that $j \leq k$ if there exists a morphism of coverings $Y_{j} \arr Y_{k}$. So $J$ becomes a cofiltered set. Deligne's absolute groupoid is $\widehat{P}_{X} \double X$, where $\widehat{P}_{X} \eqdef \projlim_{j \in J}P_{Y_{j}}$, and Deligne's relative groupoid is $\widehat{P}^0_{X} \eqdef \projlim_{j \in J}P^{0}_{Y_{j}}$.

\begin{definition}
\label{def:Deligne-fund-gerbe}
	Let $X$ be a geometrically connected algebraic stack over $\kappa$. We define Deligne's absolute fundamental gerbe $\Pi_X^\rD$ (respectively Deligne's relative fundamental gerbe $\Pi_{X/\kappa}^\rD$) as the stack associated with the groupoid $\widehat{P}_{X} \double X$ (respectively $\widehat{P}^0_{X} \double X$).
\end{definition}

Notice that, by construction, we are given natural morphisms $X \arr \Pi_{X/\kappa}^\rD \arr \Pi_{X}^\rD $. Since $\Pi_{X/\kappa}^\rD$ is an étale profinite gerbe, the composite $X \arr \Pi_{X/\kappa}^\rD$ induces a morphism $\etuniv X \arr \Pi_{X/\kappa}^\rD$.

\begin{theorem}	\label{rmk:compare-Deligne}
Assume that $X$ is \infl. The natural morphism $\etuniv X \simeq \Pi_{X/\kappa}^\rD$ is an isomorphism.
\end{theorem}

\begin{proof}
We start with a useful lemma.

\begin{lemma}
	\label{lem:con-Nori-red}
Let $Y \arr X$ be a connected étale Galois cover and $G \eqdef \aut_{X} Y$. The morphism $X \arr [X/P^{0}_{Y}]$ is Nori-reduced.
\end{lemma}
\begin{proof}
	Suppose that $X \arr [X/P^{0}_{Y}]=\mathcal G$ factors trough a representable morphism $\mathcal G' \to \mathcal G$, where $\mathcal G'$ is a finite gerbe. Since the diagonal of $\mathcal G$ is étale, so is the diagonal of $\mathcal G'$, and we deduce that both morphisms $X\times_{\mathcal G} X \to X\times_S X$ and $X\times_{\mathcal G'} X \to X\times_S X$ are also étale. Hence the natural morphism $X\times_{\mathcal G'} X \to X\times_{\mathcal G} X$ is étale. Since $\mathcal G' \to \mathcal G$ is representable, this morphism is an immersion, since $X\times_{\mathcal G} X=P_Y^0$ is connected, this must be an isomorphism.
\end{proof}

Hence with the notations of the lemma, the factorization $X \arr [X/P^{0}_{Y}]$ of $X \arr \cB_{\kappa}G$ is the canonical factorization of Lemma~\ref{lem:factor-through-Nori-reduced}. This defines a lax 2-functor $J \arr I^{\et}$, which induces a homomorphism of preordered sets $\lambda\colon J \arr \overline{I^{\et}}$, where $\overline{I^{\et}}$ is the preordered set associated with $I^{\et}$, as in Remark~\ref{rmk:viewpoint-groupoids}, such that for each $j \in J$ we have a canonical isomorphism $P^{0}_{Y_{j}} \simeq X \times_{\Gamma_{\lambda(j)}} X$. We need to show that the induced morphism $\projlim P^{0}_{Y_{j}} = \projlim (X \times_{\Gamma_{\lambda(j)}} X)$ into $X \times_{\etuniv X}X$, which, by Remark~\ref{rmk:viewpoint-groupoids} coincides with $\projlim_{\overline{I^{\et}}}(X \times_{\Gamma_{i}} X)$, is an isomorphism. In fact, let us show that the image of $J$ is cofinal in $\overline{I^{\et}}$; in other words, we need to show that given a Nori-reduced morphism $X \arr \Gamma$, where $\Gamma$ is a finite étale gerbe, there exists a finite group $G$ and a representable morphism $\Gamma \arr \cB_{\kappa}G$, such that the composite $X \arr \cB_{\kappa}G$ corresponds to a connected $G$-cover of $X$.

Since $\Gamma$ is finite and étale, there is a finite separable extension $\kappa'$ of $\kappa$ and a morphism $\spec \kappa' \arr \Gamma$, which is an étale covering; call $n$ its degree. The covering above corresponds to a representable morphism $\Gamma \arr \cB_{\kappa}\rS_{n}$ (where $\rS_{n}$ is the symmetric group on $n$ letters). Consider the $\rS_{n}$-torsor $P \arr X$ corresponding to the composite $X \arr \Gamma \arr \cB_{\kappa}\rS_{n}$, and take a connected component $Y \subseteq P$; this is a connected $G$-torsor for a subgroup $G \subseteq \rS_{n}$. There is a commutative diagram
   \[
   \xymatrix{
   &&\cB_{\kappa}G\ar[d]\\
   X\ar[r]\ar[urr]& \Gamma \ar[r] \ar@{-->}[ur] & \cB_{\kappa}\rS_{n}\,;
   }
   \]
we need to show that we can insert a dashed arrow into it. Consider the morphism $X \arr \Gamma \times_{\cB_{\kappa}\rS_{n}}\cB_{\kappa}G$ induced by the diagram above: since $X$ is \infl, the scheme-theoretic image $\Delta$ of $X$ is a finite gerbe. The projection $\Delta \arr \Gamma$ is representable, since $\cB_{\kappa}G$ is representable over $\cB_{\kappa}\rS_{n}$; since $X \arr \Gamma$ is Nori-reduced, the morphism $\Delta \arr \Gamma$ must be an isomorphism. The inverse $\Gamma \arr \Delta$, followed by the embedding $\Delta \subseteq \Gamma \times_{\cB_{\kappa}\rS_{n}}\cB_{\kappa}G$ and the projection $\Gamma \times_{\cB_{\kappa}\rS_{n}}\cB_{\kappa}G \arr \cB_{\kappa}G$ gives the required morphism.
\end{proof}

\section{The fundamental gerbe and the section conjecture}\label{sec:section}

As an application of the formalism in Section~\ref{sec:etale}, let us show that the rational points of the gerbe $\etuniv X$ have a natural interpretation in terms of sections of Grothendieck's fundamental exact sequence. Let us first recall what this means.

\begin{theorem}[\cite{grothendieck_revetements-etales-groupe-fondamental} IX, Théorème 6.1]
   \label{thm:fund-ex-seq}
   Let $X/\kappa$ be a quasi\dash compact, quasi\dash separated and geometrically connected algebraic stack, and fix a geometric point $\overline x:\spec  \Omega\rightarrow X$.  Then the following sequence is exact:
 \begin{equation}
	 1 \rightarrow \pi_1(X_{\overline \kappa},\overline x)\rightarrow  \pi_1(X,\overline x) \rightarrow G_\kappa\rightarrow 1
   \end{equation}
   where $G_\kappa$ denotes the absolute Galois group of $\kappa$ relative to the separable closure of $k$ in $\Omega$.
\end{theorem}
By functoriality, a rational point of $X$ induces a section of this exact sequence, well defined up to conjugacy by an element of $\pi_1(X_{\overline \kappa},\overline x)$. Denoting by $\IsCl$ the set of isomorphism classes of a small category, and by $\homext_{G_\kappa}(G_\kappa,\pi_1(X,\overline x))$ this set of equivalence classes of sections, we thus obtain an application: 
\[s_X: \IsCl(X(\kappa)) \rightarrow \homext_{G_\kappa}(G_\kappa,\pi_1(X,\overline x))  \]

Grothendieck's famous section conjecture is stated as follows:

 \begin{conjecture}[\cite{grothendieck_brief-faltings}]
   \label{conj:section}
If $X$ is a proper, smooth and geometrically connected curve of genus at least $2$ over a finitely generated extension $\kappa$ of $\mathbb Q$, the application $s_X$ is one to one. 
\end{conjecture}

As we will explain now, this conjecture can be stated in our terms by saying that the natural morphism $X\to \Pi_{X/\kappa}^\rD$ induces a bijection of isomorphism classes of $\kappa$-rational points. Rather than working with isomorphism classes, it is natural to consider the category $X(\kappa)$ and the category $\HomExt_{G_\kappa}(G_\kappa,\pi_1(X,\overline x))$ whose objects are sections, and morphisms are given by conjugacy by elements of $\pi_1(X_{\overline \kappa},\overline x)$, in the natural way.

\begin{proposition}
	\label{prop:comp_section}
	Let $X/\kappa$ be a quasi-compact, quasi-separated and \infl algebraic stack, and fix a geometric point $\overline x:\spec  \Omega\rightarrow X$.
	There is a (non canonical) equivalence of categories $\Pi_{X/\kappa}^\rD(\kappa) \to \HomExt_{G_\kappa}(G_\kappa,\pi_1(X,\overline x))$ that composed with the canonical functor 
	$X(\kappa)\to \Pi_{X/\kappa}^\rD(\kappa)$ is a lifting of $s_X$ at the level of categories. \end{proposition}	

\begin{proof}
This will follow from the next lemma.
\begin{lemma}
	\label{lem:hom-seq-Del}
The natural diagram:
   \[
   \xymatrix{
   \Pi_{X/\kappa}^\rD \ar[d] \ar[r] &  \Pi_X^\rD \ar[d] \\
   \spec \kappa \ar[r] & \Pi_{\spec \kappa}^\rD
   }
   \]
   is cartesian.\qedhere
\end{lemma}

\end{proof}
Indeed, let us admit this for a while. 

If we fix a geometric point $\overline x:\spec  \Omega\rightarrow X$, then by definition of Deligne absolute fundamental gerbe, there is a canonical isomorphism $\aut_{\Pi_X^\rD}(\overline x)\simeq \pi_1(X,\overline x)$, and it follows that there is a compatible isomorphism $\aut_{\Pi_{X/\kappa}^\rD}(\overline x)\simeq \pi_1(X_{\overline \kappa},\overline x)$. In other words, the fundamental exact sequence can be interpreted as the exact sequence of $\aut$-groups associated with the fibration at $\overline x$.

We note that we can safely replace $X$ by its étale fundamental gerbe: indeed, if $\overline x:\spec  \Omega\rightarrow X$ is a geometric point, it induces a morphism of the corresponding fundamental exact sequences, that is in fact an isomorphism of exact sequences. Since $\Pi_{X/\kappa}^\rD$ is a profinite étale gerbe, the morphism $\Pi_{X/\kappa}^\rD\to \Pi_{\Pi_{X/\kappa}^\rD/\kappa}^\rD$ is an isomorphism, and the lemma implies then that the morphism $\Pi_{X}^\rD\to \Pi_{\Pi_{X}^\rD/\kappa}^\rD$ is an isomorphism. 

So let us assume $X=\mathcal G$ is a profinite étale gerbe, and fix a point $x\in \mathcal G(S)$. It induces a canonical equivalence $ \mathcal G(S)\simeq \tors_\kappa \aut_\kappa x$, the category of torsors over $\kappa$ under the $\kappa$-group  $\aut_\kappa x$. It induces also a canonical section of the exact sequence of $\aut$-groups associated with the fibration at the corresponding geometric point $\overline x$. This section induces, in turn, as is well known, a canonical equivalence $\HomExt_{\aut_{\Pi_S}(\overline x)}(\aut_{\Pi_S}(\overline x),\aut_{\Pi_{\mathcal G}}(\overline x))\simeq \tors_{\aut_{\Pi_S}(\overline x)-\mathrm{sets}}(\aut_{\mathcal G}(\overline x))$, the category of torsors under the $\aut_{\Pi_S}(\overline x)$-group $\aut_{\mathcal G}(\overline x)$ in the category of $\aut_{\Pi_S}(\overline x)$-sets. Hence we get the (non canonical) equivalence we needed.

\begin{proof}[Proof of Lemma \ref{lem:hom-seq-Del}]
 Let $Y\to X$ be a connected Galois cover, with Galois group $G$. Let $\kappa'/\kappa$ be the largest separable extension such that $X_{\kappa'}\to X$ is a quotient of $Y\to X$, and let $H$ be its Galois group. It is enough to show that the following diagram is cartesian:   
 \[
   \xymatrix{
   [X/P^{0}_{Y}] \ar[d] \ar[r] &  \cB_{\kappa}G \ar[d] \\
   \spec \kappa \ar[r] & \cB_{\kappa}H
   }
   \]
   Denote by $\mathcal G\eqdef \spec \kappa\times_{\cB_{\kappa}H}\cB_{\kappa}G$, this is clearly a gerbe, since $G\to H$ is an epimorphism. Because of uniqueness of Nori-reduced factorization (Lemma \ref{lem:factor-through-Nori-reduced}) and the fact that $X\to [X/P^{0}_{Y}]$ is Nori-reduced (Lemma \ref{lem:con-Nori-red}), it is enough to show that the natural morphism $X\to \mathcal G$ is Nori-reduced. So assume that $X \to \mathcal G$ factors trough a representable morphism $\mathcal G' \to \mathcal G$, where $\mathcal G'$ is a finite gerbe. Let $U\eqdef \mathcal G'\times _{\cB_{\kappa}G}\spec \kappa$, this is a finite $\kappa$-scheme, endowed with a transitive action of $G$, and a $G$-equivariant map $Y\to U$. Since $Y$ is non-empty, $Y\to U$ is an epimorphism. The scheme $U$ is also endowed with a natural morphism $U\to \spec \kappa '=\mathcal G\times _{\cB_{\kappa}G}\spec \kappa$. Since $Y$ is geometrically connected over $\kappa '$ by construction, so is $U$. Moreover the morphism $U\to \mathcal G'$ is étale, and so by Proposition \ref{prop:gerbe<->geom} the scheme $U$ is reduced, hence is étale. We conclude that $U\to  \spec \kappa'$ is an isomorphism, and so is $\mathcal G' \to \mathcal G$.
\end{proof}

So according to Proposition \ref{prop:comp_section}, we can reformulate Conjecture \ref{conj:section} into the following one, that can be stated without choosing a geometric point:

 \begin{conjecture}   \label{conj:section_reform}
If $X$ is a proper, smooth and geometrically connected curve of genus at least $2$ over a finitely generated extension $\kappa$ of $\mathbb Q$, the natural morphism $X\to \Pi_{X/\kappa}^\rD$ induces a bijection of isomorphism classes of $\kappa$-rational points. 
\end{conjecture}
The injectivity is known as a consequence of the Mordell-Weil theorem (see for instance \cite{stix_cuspidal_2008}, Appendix B). This goes trough over a function field if one uses the Nori fundamental gerbe:

\begin{proposition}\label{prop:section_function_field}
	Let $\kappa$ be a field that is finitely generated over its prime subfield, and $X/\kappa$ a proper, smooth, and geometrically connected curve of genus at least $1$. Then the natural functor
	\[ X(\kappa) \to \Pi_{X/\kappa}(\kappa)\]
	is injective on isomorphism classes.
\end{proposition}

\begin{proof}
	Let us denote by $s_X(x)$ the image of a rational point $x\in X(\kappa)$ in $\Pi_{X/\kappa}(\kappa)$. We have to show that if two rational points $x,y \in  X(\kappa)$ give rise to isomorphic sections $s_X(x)\simeq s_X(y)$, then $x=y$. Since we can embed $X$ into an abelian variety $A/\kappa$, and $\Pi_{X/\kappa}$ is covariant in $X$, it is enough to show the corresponding statement for $A$. So let $a,b \in A(\kappa)$ such that $s_A(a)\simeq s_A(b)$. Nori's generalization of Lang-Serre theorem \cite{nori_fundamental_1983} and Remark \ref{rmk:gerbe-vs-group} imply that
	\[ \Pi_{A/S}\simeq \projlim_n \cB_\kappa A[n] \, .\]
	Since isomorphism classes in $\cB_\kappa A[n](\kappa)$ are in one to one correspondence with $\h^1(\kappa, A[n])$, it follows, for any non negative integer $n$, that $a-b$ lies in $\ker(A(\kappa)\to \h^1(\kappa, A[n]))=nA(\kappa)$, in other words $a-b$ is divisible in $A(\kappa)$. But the Mordell-Weil theorem, that asserts that $A(\kappa)$ is of finite type, holds for $A/\kappa$ (\cite{lang_rational_1959}, Theorem~1, see also \cite{conrad_chows_2006}, Corollary 7.2).
\end{proof}

\begin{remark}
	It is unclear whether Proposition \ref{prop:section_function_field} remains true when one replaces $\Pi_{X/\kappa}$ by $\Pi_{X/\kappa}^{\et}$. The critical point is to identify $\ker(A(\kappa)\to \h^1(\kappa, A[n]^\et))$, which does not seem to be an easy task. 
	
Also, when $k$ is a finite field, $\etuniv X(\kappa)$ is never empty, because the Galois group of $k$ is free. However, there examples of hyperbolic curves over a finite field with no rational points. Hence for finitely generated fields of positive characteristic $X(\kappa) \to \Pi_{X/\kappa}(\kappa)$ is in general not surjective on isomorphism classes, as and the section conjecture, as originally stated, fails. 
\end{remark}

This leads naturally to the following extension of the Grothendieck section conjecture to arbitrary characteristic.

 \begin{conjecture}   \label{conj:section_extend}
If $X$ is a proper, smooth and geometrically connected curve of genus at least $2$ over a finitely generated field $\kappa$, the natural morphism $X\to \univ X$ induces a bijection of isomorphism classes of $\kappa$-rational points.
\end{conjecture}

\section{The tame fundamental gerbe}\label{sec:tame}

Unfortunately, the étale fundamental gerbe does not seem to have a natural tannakian interpretation purely in terms of vector bundles. There is another quotient of $\univ X$, however, that does.

\begin{definition}
A finite stack $\Gamma$ over $\kappa$ is \emph{tame} if the functor $\coh \Gamma \arr \vect_{\kappa}$ given by $F \arrto \H^{0}(\Gamma, F)$ is exact.
\end{definition}

A finite stack $\Gamma$ has a moduli space $\pi\colon \Gamma \arr M$, which is finite over $\kappa$, hence $\Gamma$ is tame if and only if $\pi_{*}\colon \coh \Gamma \arr \coh M$ is exact; so this is a particular case of the notion of tame stack in \cite{dan-olsson-vistoli1}.

From \cite[Theorem~3.2]{dan-olsson-vistoli1} we obtain the following.

\begin{proposition}\label{prop:char-tame-stacks}
A finite stack $\Gamma$ is tame if and only if for any object $\xi$ in $\Gamma(\Omega)$, where $\Omega$ is an algebraically closed extension of $\kappa$, the group scheme $\underaut_{\Omega}\xi \arr \spec \Omega$ is linearly reductive.

\end{proposition}

Recall, again from \cite{dan-olsson-vistoli1}, that a finite group scheme $G$ over a field is linearly reductive if every finite-dimensional representation of $G$ is a sum of irreducible representations. Over an algebraically closed field $\Omega$, the finite group scheme $G$ is linearly reductive if and only if it is a semidirect product $H \ltimes D$, where $H$ is a finite constant group of order prime to the characteristic of $\Omega$, and $D$ is a finite diagonalizable group (\cite[Proposition~2.10]{dan-olsson-vistoli1}). Of course in characteristic~$0$ every finite group scheme is linearly reductive, so every finite stack is tame.

\begin{proposition}\call{prop:properties-tame}\hfil
\begin{enumeratea}

\itemref{a} If $\Delta \arr \Gamma$ is a representable morphism of tame stacks, and $\Gamma$ is tame, then so is $\Delta$.

\itemref{b} If $\Delta_{1} \arr \Gamma$ and $\Delta_{2}\arr \Gamma$ are morphisms of tame finite stacks, the fibered product $\Delta_{1} \times_{\Gamma} \Delta_{2}$ is also tame.

\itemref{c} If $\kappa'$ is an extension of $\kappa$, then $\Gamma_{\kappa'}$ is tame over $\kappa'$ if and only if $\Gamma$ is tame over $\kappa$.

\itemref{d} If $G$ is a finite group scheme over $\kappa$, then $\cB_{\kappa}G$ is tame if and only if $G$ is linearly reductive.

\itemref{e} A profinite gerbe $\Phi$ is tame if and only if the category $\rep \Phi$ is semisimple.
\end{enumeratea}
\end{proposition}

\begin{proof}
Part \refpart{prop:properties-tame}{d} follows immediately from the definitions.

Parts \refpart{prop:properties-tame}{a} and \refpart{prop:properties-tame}{b} are easy consequences of the following facts.

\begin{enumeratea}

\item A subgroup of a finite linearly reductive group scheme is linearly reductive \cite[Proposition 2.5~(a)]{dan-olsson-vistoli1}.

\item A product of finite linearly reductive group schemes is linearly reductive \cite[Proposition 2.5~(c)]{dan-olsson-vistoli1}.

\end{enumeratea}

\refpart{prop:properties-tame}{c} is a particular case of \cite[Corollary~3.4]{dan-olsson-vistoli1}.

\refpart{prop:properties-tame}{e} is easy and left to the reader.
\end{proof}

Let $X$ be an \infl algebraic stack, $\univ X = \projlim_{i} \Gamma_{i}$ as above. Consider the full subcategory $I^{\rm tame}$ of $I$ whose objects consist of Nori-reduced morphisms $X \arr \Gamma_{j}$ in which $\Gamma_{j}$ is tame. Since the fibered product of two tame gerbes is a tame stack, we have that $I^{\rm tame}$ is a 2-cofiltered category.

\begin{definition}
Let $X$ be an algebraic stack over $\kappa$. 
The tame fundamental gerbe of $X$ is the profinite gerbe $\tuniv X \eqdef \projlim_{j \in I^{\rm tame}} \Gamma_{j}$.
\end{definition}

Again, there is a natural morphism $\univ X \arr \tuniv X$. It is easy to see that for any tame finite stack $\Gamma$ over $\kappa$, the natural functor
   \[
   \hom(\tuniv X, \Gamma) \arr \hom(X, \Gamma)
   \]
is an equivalence. Also, the morphism $X \arr \tuniv X$ is universal among all maps from $X$ to a tame profinite gerbe.

\section{The tannakian interpretation of the étale fundamental gerbe}\label{sec:etale-tannakian}

Let $X$ be a connected algebraic stack. We endow it with the finite étale topology, where coverings are given by surjective families of finite étale morphisms. We denote by $X\fet$ the corresponding site. Let $\LC(X\fet,\kappa)$ be the category of local systems of finite dimensional $\kappa$-vector spaces. This is a Tannaka category. The aim of this section is to relate it to $\efin X$, when $X$ is defined over $\kappa$, inflexible and \psp.

We can describe the objects of $\LC(X\fet,\kappa)$ in terms of gerbes in the following way. Since $X$ is connected, we can stick to covers consisting of a single finite étale morphism $Y\rightarrow X$, that we can suppose Galois, with Galois group $G$. A local system $V$ in $\LC(X\fet,\kappa)$ trivialized by such a $G$-torsor $Y\to X$ corresponds to a morphism $X\to \cB_\ZZ G$, and a local system $V_0$ on $\cB_\ZZ G$, which is nothing else than a finite dimensional $\kappa$-linear representation $\rho_0\colon G\to \GL(V_0)$ of the constant group $G$.

In view of this, it is clear that the gerbe associated with the Tannaka category $\LC(X\fet,\kappa)$ is canonically isomorphic to Deligne's absolute fundamental gerbe $\Pi_X^\rD$ (Definition \ref{def:Deligne-fund-gerbe}).

\begin{definition}
	Let $X$ be a connected algebraic stack over $\kappa$. The Riemann-Hilbert functor $\RH_{X/\kappa}:\LC(X\fet,\kappa)\rightarrow \vect X$ is defined by the formula $\RH_{X/\kappa}(V)=\mathcal O_X\otimes_\kappa V$. 
\end{definition}
In this description, it is implicit that we use descent theory to get a vector bundle on $X$. In other terms, if $V$ comes from a morphism $X\to B_\ZZ G$ and a representation $\rho_0:G\to \GL(V_0)$ as above, we can use the diagram:
   \[
   \xymatrix{
   &\spec\kappa \ar[r]\ar[d] &\spec\ZZ\ar[d] \\
   X \ar[r] &B_\kappa G \ar[r] &B_\ZZ G
      }
   \]
   to descend the trivial bundle $\mathcal O_{\spec\kappa}\otimes_\kappa V_0$ to $B_\kappa G$, and then pull it back to $X$, to get $\RH_{X/\kappa}(V)$. This construction is independent (up to canonical isomorphism) of the choice of the pair $(X\to B_\ZZ G,\rho_0:G\to \GL(V_0))$. Moreover, this description makes it clear that, when $X$ is inflexible and pseudo-proper, the functor $\RH_{X/\kappa}$ factors trough $\efin X$ (see Lemma \refall{lem:reduced-finite-stacks}{a}). It is also clear that the functor $\RH_{X/\kappa}$ is functorial in $X/\kappa$. 

  To state our result we need two additional definitions.
\begin{definition}
	  Let $X/\kappa$ be an \infl and \psp algebraic stack. We denote by $\efin^\et X$ the full subcategory of $\efin X$ whose objects are essentially finite locally free sheaves with étale holonomy gerbe.
  \end{definition}
  In other words, $\efin^\et X$ is the category of representations of $\etuniv X$. It is clear that the functor $\RH_{X/\kappa}$ factors trough $\efin^\et  X$.

  \begin{definition}[\cite{milne_quotients_2007}] \label{def:quotient-tannaka}
	  Let $\mathcal C$ be a Tannaka category over $\kappa$, and $\mathcal B$ be a Tannaka subcategory. We say that a tensor functor $F:\mathcal C \to \mathcal D$ to a Tannaka category $\mathcal D$ identifies $\mathcal D$ with the quotient of $\mathcal C$ by $\mathcal B$ if:
\begin{enumerate}
	\item the objects of $\mathcal C$ whose image by $F$ is trivial are exactly those of $\mathcal B$,
	\item any object of $\mathcal D$ is a subquotient of an object in the image of $F$.
\end{enumerate}
  \end{definition}

\begin{theorem}
	\label{thm:tannaka-int-etale}
 Let $X/\kappa$ be a \infl and \psp stack, and denote by $s:X\rightarrow \spec \kappa$ the structural morphism. Then the functor $s^*:\LC (  ( \spec \kappa)\fet,\kappa)\to\LC(X\fet,\kappa)$ is fully faithful and the functor $\RH_{X/\kappa}$ identifies $\efin^\et X$ with the quotient of $\LC(X\fet,\kappa)$ by $\LC( (\spec \kappa)\fet,\kappa)$.

\end{theorem}

\begin{proof}
	According to Proposition \refall{prop:geometric-properties-infl}{a}, the stack $X/\kappa$ is geometrically connected, and one deduces that, if $\kappa_X$ denotes the $1$-dimensional constant sheaf on $X$, we have $s_* \kappa_X=\kappa_{\spec \kappa}$. It now follows from the projection formula that the functor $s^*$ is fully faithful.

	To check the first condition of Definition \ref{def:quotient-tannaka}, we have to prove that, given an object $V$ of $\LC(X\fet,\kappa)$, its image $\RH_{X/\kappa}(V)$ is trivial (that is, a free vector bundle) if and only if there exists an object $W$ of $\LC( (  \spec \kappa)\fet,\kappa)$ such that $V\simeq s^*W$. The ``if'' direction is clear, because of the functoriality of $\RH_{X/\kappa}$ in $X/\kappa$. To prove the ``only if'' direction, let us start from a local system $V$ in $\LC(X\fet,\kappa)$ such that $\RH_{X/\kappa}(V)$ is a trivial vector bundle. We can assume that $V$ comes from a morphism $X \rightarrow B_\kappa G$, and a local system $V_0$ on $B_\kappa G$, as above. Since $X$ is \infl, we can use Lemma \ref{lem:factor-through-Nori-reduced} to factor $X \rightarrow B_\kappa G$ trough a Nori-reduced morphism to a finite gerbe  $f\colon X \rightarrow \Gamma$. Let $V_1$ be the pull-back of $V_0$ along  the morphism $\Gamma \to B_\kappa G$. Using the functoriality of $\RH_{X/\kappa}$ in $X/\kappa$ again, and the fact that the functor $f^*: \vect \Gamma \to \vect X$ is fully faithful (Lemmas \ref{lem:pullback-faithfully-flat}, \ref{lem:Nori-reduced-coh-flat}), we see that $\RH_{\Gamma/\kappa}(V_1)$ is trivial. So we can finally assume that $X=\Gamma$ is a finite gerbe, and $V=V_1$.
\begin{lemma}
	Let $\Gamma/\kappa$ be a finite gerbe, $\kappa^{\rm sep}$ a separable closure of  $\kappa$, and $V$ an object of\/ $\LC(\Gamma\fet,\kappa)$. Then $V$ descends to an object of  $\LC((  \spec \kappa)\fet,\kappa)$ if and only if $V_{|\Gamma_{\kappa^{\rm sep}}}$ is constant.	
\end{lemma}
\begin{proof}
	This is a straightforward consequence of descent theory, since we can see $(\Gamma \to \spec \kappa, \spec \kappa^{\rm sep}\to \spec \kappa)$ as a cover of $\spec \kappa$.
\end{proof}
So the last thing to do to prove the first condition of Definition \ref{def:quotient-tannaka} is to show that if $\RH_{\Gamma/\kappa}(V)$ is trivial, then $V_{|\Gamma_{\kappa^{\rm sep}}}$ is constant. This is a consequence of  the functoriality of $\RH_{X/\kappa}$ in $X/k$ and of the following obvious lemma.
\begin{lemma}
	The functor $\RH_{\Gamma_{\kappa^{\rm sep}}/\kappa^{\rm sep}}:\LC({\Gamma_{\kappa^{\rm sep}}}\fet,\kappa^{\rm sep})\to \vect\Gamma_{\kappa^{\rm sep}}$ is an equivalence.
\end{lemma}
\begin{proof}
	This is clear, since $\Gamma_{\kappa^{\rm sep}}\simeq B_{k^{\rm sep}}G$ for a finite constant group $G$, and both categories then identify with the category of finite dimensional representations of $G$ with values in $\kappa^{\rm sep}$.
\end{proof}
Note that since $V\otimes_\kappa \kappa^{\rm sep}$ is constant, the same holds for $V_{|\Gamma_{\kappa^{\rm sep}}}$, and this finishes the proof of the first condition of Definition \ref{def:quotient-tannaka}.

We now prove the second condition. Let $E$ be an object of $ \efin^\et X=\rep \etuniv X$. According to Remark \ref{rmk:compare-Deligne}, we can choose a connected Galois cover $Y\to X$ of group $G$ so that $E$ comes from a representation of the groupoid $[X/P^{0}_{Y}]$ along the Nori-reduced morphism $X \arr [X/P^{0}_{Y}]$.  This morphism is the canonical factorization of the given morphism $X\arr \cB_\kappa {G}$, in particular, the morphism $f\colon  [X/P^{0}_{Y}]\arr \cB_\kappa{G}$ is representable.

\begin{lemma}
Let $f\colon \Gamma \arr \Delta$ be a representable morphism of finite gerbes.  Then every representation of $\Gamma$ is a quotient of a representation of\/ $\Delta$.
\end{lemma}

\begin{proof}
	It is enough to show that $f$ is affine, since then for any vector bundle $E$ on $\Gamma$, the canonical morphism $f^*f_* E\to E$ is an epimorphism.
 We may base change so that $\Gamma(\spec\kappa)$ is not empty. Then $f\colon \Gamma \arr \Delta$ can be identified with a morphism $\phi\colon \cB_{\kappa}H \arr \cB_{\kappa}G$ induced by an injective homomorphism $H \arr G$ of finite $\kappa$-group schemes. But then the base change of $\phi$ in the chart $\spec \kappa \to \cB_{\kappa}G$ is $G/H\arr \spec \kappa$, which is affine, since $H$ and $G$ are finite.
\end{proof}

So the vector bundle $E$, seen as a representation of $[X/P^{0}_{Y}]$, is a quotient of $f^*f_*E$, where $f$ is the natural morphism: $f\colon [X/P^{0}_{Y}]\arr \cB_\kappa{G}$ . But $f_*E$ is trivialized by the $G$-torsor $\spec \kappa \to \cB_{\kappa}G$, hence corresponds to a representation of the finite constant group $G$ with values in finite dimensional $\kappa$-vector spaces. This gives rise to a local system on $X$ which is sent to $f_*E$ by $\RH_{X/\kappa}$. This proves the second condition of Definition \ref{def:quotient-tannaka}, hence the Theorem.
\end{proof}

\section{The tannakian interpretation of the tame fundamental gerbe}\label{sec:tame-tannakian}

Let us fix an \infl \psp algebraic stack $X$ over $\kappa$. Then we have an equivalence of Tannaka categories $\rep \univ X \arr \efin X$.

\begin{definition}
A locally free sheaf $E$ on $X$ is \emph{tamely finite} when it is finite, and all the indecomposable components of all the tensor powers $E^{\otimes n}$ are irreducible in $\efin X$.

We denote by $\tfin X$ the full subcategory of $\efin X$ consisting of tamely finite objects.
\end{definition}

It is easy to see that $\tfin X$ is an exact abelian subcategory of $\efin X$; however, it does not seem obvious to us that it is a tannakian subcategory, that is, that the tensor product of two tamely finite sheaves is tamely finite.

Clearly, if $\Gamma$ is a tame finite gerbe every object of $\rep \Gamma$ is tamely finite; hence the pullback $\rep \tuniv X \arr \rep\univ X$, which according to Lemma \ref{lem:pullback-faithfully-flat} is fully faithful, gives an equivalence of $\rep \tuniv X$ with a tannakian semisimple subcategory of $\rep \univ X$, which is contained in $\tfin X$. The following theorem says that this is an equivalence.

\begin{theorem}
The pullback $\rep \tuniv X \arr \efin X$ induces an equivalence of the Tannaka category $\rep \tuniv X$ with $\tfin X$.

In particular, $\tfin X$ is a tannakian subcategory of $\efin X$.
\end{theorem}

It follows from the Theorem that $\tfin X$ is the largest tannakian semisimple subcategory of $\efin X$.

\begin{proof}
The proof of the Theorem is based on the following fact. Let $V$ be a representation of degree $r$ of a finite gerbe $\Gamma$, corresponding to a morphism $\Gamma \arr \cB_{\kappa}\GL_{r}$. We say that $V$ is \emph{faithful} if the morphism is representable. Suppose that $\kappa'$ is extension of $\kappa$, such that $\Gamma(\kappa') \neq \emptyset$; then $\Gamma_{\kappa'} \simeq \cB_{\kappa'}G$ for a finite group scheme $G$ over $\kappa'$. The pullback $V_{\kappa'}$ of $V$ to $\Gamma_{\kappa'}$ comes from a representation $G \arr \GL_{r}$; then $V$ is faithful if and only if the representation $G \arr \GL_{r}$ is faithful, in the sense that it has trivial kernel.

\begin{lemma}\label{lem:faithful->tame}
Suppose that a finite gerbe $\Gamma$ has a faithful tamely finite representation. Then $\Gamma$ is tame.
\end{lemma}

\begin{proof}
Let $V$ be a faithful representation of degree~$r$ of $\Gamma$; this corresponds to a morphism $f\colon \Gamma \arr \cB_{\kappa}\GL_{r}$. Notice that $f$ is affine; for this we can extend the base field and assume that $\Gamma = \cB_{\kappa}G$ for a certain finite group scheme $G$, so that $V$ corresponds to a faithful representation $G \arr \GL_{r}$. In this case the fiber of $f$ over $\spec \kappa$ is the quotient $\GL_{r}/G$; and since the quotient of an affine variety by a finite group scheme is affine, we have the result.

Let $\kappa'$ be a finite extension of $\kappa$ such that $\Gamma(\kappa') \neq \emptyset$, and choose a morphism $\rho\colon \spec \kappa' \arr \Gamma$. By Lemma \refall{lem:reduced-finite-stacks}{c} every representation of $\Gamma$ is contained in a sum of copies of $E \eqdef \rho_{*}\cO_{\spec \kappa'}$; hence it is enough to prove that $E$ is semisimple. Set $\sigma \eqdef f\rho \colon \spec \kappa' \arr \cB_{\kappa}\GL_{r}$. Consider the adjunction homomorphism $f^{*}f_{*}E \arr E$; this is surjective, because $f$ is affine, so it is enough to show that the infinite dimensional representation $f^{*}f_{*}E = f^{*}\sigma_{*}\cO_{\spec \kappa'}$ is semisimple.

Notice that any two morphisms $\spec \kappa' \arr \cB_{\kappa}\GL_{r}$ are isomorphic; hence $\sigma \simeq \pi\alpha$, where $\alpha\colon \spec \kappa' \arr \spec \kappa$ is the canonical morphism and $\pi\colon \spec \kappa \arr \cB_{\kappa}\GL_{r}$ corresponds the trivial $\GL_{r}$-torsor on $\cB_{\kappa}\GL_{r}$. If $d$ denotes the degree of the extension $\kappa'/\kappa$, then $\alpha_{*}\cO_{\spec \kappa'} \simeq \cO_{\spec \kappa}^{\oplus d}$, so $\sigma_{*}\cO_{\spec \kappa'} \simeq (\pi_{*}\cO_{\spec k})^{\oplus d}$; so it is enough to prove that $V \eqdef f^{*}\pi_{*}\cO_{\spec k}$ is semisimple. The representation $\pi_{*}\cO_{\spec k}$ corresponds to the standard regular representation of $\GL_{r}$, that is, to the action of $\GL_{r}$ on the vector space $k[\GL_{r}]$ induced by right translation. Let $W$ be the tautological representation of $\GL_{r}$ of degree~$r$; the pullback of $W$ to $\Gamma$ is exactly $V$. On the other hand we have that $k[\GL_{r}]$ is a quotient of the representation
   \[
   \bigoplus_{m, n \geq 0} \det ( W^{\vee})^{\otimes m}
   \otimes \sym^{n}(W \otimes W^{\vee})\,,
   \]
from which we obtain that $f^{*}\pi_{*}\cO_{\spec k}$ is a quotient of 
\[
    \bigoplus_{m, n \geq 0} \det ( V^{\vee})^{\otimes m}
   \otimes \sym^{n}(V \otimes V^{\vee})\,.
   \]
Since each of the summands $\det (V^{\vee})^{\otimes m} \otimes \sym^{n}(V \otimes V^{\vee})$ is a quotient of a tensor product of tensor powers of $V$ and $V^{\vee}$, and all these tensor powers are semi-simple, because $V$, and hence $V^{\vee}$, is tamely finite. This concludes the proof.
\end{proof}

Let us proceed with the proof of the theorem. According to the discussion above, it is enough to show that every tamely finite sheaf $E$ on $X$ is in the essential image of $\rep \tuniv X$. Choose a morphism $f\colon X \arr \Gamma$ to a finite gerbe and a representation $V$ of $\Gamma$ with $f^{*}V \simeq E$. The representation $V$ corresponds to a morphism $\Gamma \arr \cB_{\kappa}\GL_{r}$, where $r$ is the rank of $E$. We claim that there is a factorization $\Gamma \arr \Delta \arr \cB_{\kappa}\GL_{r}$, where $\Delta$ is a finite gerbe and the morphism $\Delta \arr \cB_{\kappa}\GL_{r}$ is representable. Let $I_{\Gamma}$ be the inertia stack of $\Gamma$, which is a group scheme over $\Gamma$, and let $G \subseteq I_{\Gamma}$ be the kernel of the induced homomorphism of relative group schemes $I_{\Gamma} \arr \Gamma \times_{\cB_{\kappa}\GL_{r}} I_{\cB_{\kappa}\GL_{r}} \subseteq \Gamma \times \GL_{r}$; the morphism $\Gamma \arr \cB_{\kappa}\GL_{r}$ is representable if and only if $G$ is trivial. Now we can take $\Delta$ to be the rigidification $\Gamma\thickslash G$, as in \cite[Theorem~A.1]{dan-olsson-vistoli1}. Then the tamely finite sheaf $E$ is a pullback from a representation of $\Delta$, which is a tame finite gerbe, according to Lemma~\ref{lem:faithful->tame}.
\end{proof}

\section{Examples and applications}\label{sec:examples}

We conclude with some examples to illustrate the theory. We will use without comments some standard facts in the theory of Brauer--Severi varieties;  as a general reference, see \cite{artin-brauer-severi} or \cite[Chapter~5]{gille-szamuely}.

\begin{proposition}\label{prop:brauer-severi}
Let $P$ be a Brauer--Severi variety over $\kappa$. Then $\univ P = \spec \kappa$.
\end{proposition}

\begin{proof}
There exists some finite separable extension $\kappa'$ of $\kappa$ such that $P(\kappa') \neq \emptyset$, so that $P_\kappa' = \PP^{n}_{\kappa'}$. From Proposition~\ref{prop:base-change} we see that we may assume $\kappa = \kappa'$, that is, $P= \PP^{n}_{\kappa}$.  Also, we can assume that $\kappa$ is infinite (if not, pass to a separable closure).

It is enough to show that every finite bundle $E$ on $\PP^{n}_{\kappa}$ is trivial, as this implies immediately that $\efin \PP^{n}_{\kappa} = \vect_{\kappa}$, which implies the thesis, by Tannaka duality. By \cite[Theorem~3.2.1]{okonek-schneider-spindler}, it is enough to show that the restriction of $E$ to each line is trivial (in the reference given the the result is only stated over $\CC$, but the proof works over any infinite field). Since the restriction of $E$ to a line is again finite, we are reduced to the case $n =1$, which follows from Proposition~\ref{prop:properties-finite}, and Grothendieck's theorem on the structure of vector bundles on $\PP^{1}_{\kappa}$.
\end{proof}

Let us give examples of schemes $X$ over $\kappa$ in which $\univ X(\kappa) = \emptyset$. Here is a general method for producing examples. 

\begin{proposition}\label{prop:criterion-no-sections}
Let $P$ be a Brauer--Severi variety $\kappa$. Call $r$ its exponent; then $\pic P$ is generated by a sheaf of degree $r$, which we call $\cO_{P}(r)$. Let $f\colon X \arr P$ be a morphism, where $X$ is an inflexible algebraic stack. Assume that there exists a prime $p$ dividing $r$ and an invertible sheaf $\Lambda$ on $X$, such that $\Lambda^{\otimes p} \simeq f^{*}\cO_{P}(r)$. Then $\univ X(\kappa) = \emptyset$.

\end{proposition}

\begin{proof}
Let $P^{\vee}$ be the dual Brauer--Severi variety, that is, the Hilbert scheme of hyperplanes in $P$. Then $P^{\vee}$ has also exponent~$r$. Let $\Gamma \arr \spec \kappa$ be the stack, whose sections over a $\kappa$-scheme $S$ consist of invertible sheaves $E$ on $S \times P^{\vee}$, with an isomorphism $E^{\otimes p} \simeq \pr_{2}^{*}\cO_{P^{\vee}}(r)$. It is easy to check that $\Gamma$ is a gerbe banded by $\mu_{p}$. Clearly $\Gamma(\kappa) = \emptyset$, because $\cO_{P^{\vee}}(r)$ has no $p\th$ root on $P^{\vee}$.

Denote by $H \subseteq P \times P^{\vee}$ the tautological divisor; the invertible sheaf $\cO(H)$ has bidegree $(1,1)$. Notice that there is an isomorphism
   \[
   \cO(rH) \simeq \pr_{1}^{*}\cO_{P}(r) \otimes \pr_{2}^{*}\cO_{P^{\vee}}(r)\,.
   \]
Consider the morphism $f\times\id\colon X \times P^{\vee} \arr P \times P^{\vee}$, and the invertible sheaf
   \[
   E \eqdef (f \times \id)^{*}\cO(H)^{\otimes r/p} \otimes \pr_{1}^{*}\Lambda^{\vee}
   \]
on $X \times P^{\vee}$. We claim that $E^{\otimes p}$ is isomorphic to $\pr_{2}^{*}\cO_{P^{\vee}}(r)$; this gives a morphism $X \arr \Gamma$, and shows that $\univ X(\kappa) = \emptyset$.

If $F$ is a geometric fiber of the projection $\pr_{1}\colon X \times P^{\vee} \arr X$, the invertible sheaf $E^{\otimes p} \otimes \pr_{2}^{*}\cO_{P^{\vee}}(-r)$ is trivial along $F$; since $\H^{1}(F, \cO) = 0$, we have, by the standard base change theorems, that $\pr_{1*}\bigl(E^{\otimes p} \otimes \pr_{2}^{*}\cO_{P^{\vee}}(-r)\bigr)$ is an invertible sheaf on $X$, and that
   \begin{align*}
   E^{\otimes p} \otimes \pr_{2}^{*}\cO_{P^{\vee}}(-r) &=
   \pr_{1}^{*}\pr_{1*}\bigl(E^{\otimes p} \otimes \pr_{2}^{*}\cO_{P^{\vee}}(-r)\bigr)\\
   &= \pr_{1}^{*}\pr_{1*}\bigl((f \times \id)^{*}\cO(rH)
      \otimes \pr_{1}^{*}\Lambda^{\otimes - p} \otimes \pr_{2}^{*}\cO_{P^{\vee}}(-r)\bigr)\\
   &= \pr_{1}^{*}\pr_{1*}\bigl(\pr_{1}^{*}(f^*\cO_{P}(r)
      \otimes \Lambda^{\otimes - p}) \otimes \pr_{2}^{*}(\cO_{P^{\vee}}(r)
         \otimes \cO_{P^{\vee}}(-r))\bigr)\\
   &= \pr_{1}^{*}\pr_{1*} \cO_{X \times P^{\vee}}\\
   &= \cO_{X}\,.
   \end{align*}
This concludes the proof.
\end{proof}

\begin{remark}\label{rmk:weaker-condition}\hfil
	\begin{enumerate}
		\item In applying Proposition~\ref{prop:criterion-no-sections}, it is useful to notice that if $f^{*}\cO_{P}(mr)$ is a $p\th$ power in $\pic X$, where $m$ is not divisible by $p$, then $f^{*}\cO_{P}(r)$ is also a $p\th$ power. 		
		\item Using the relative Brauer group, Jakob Stix has shown the following similar result (\cite{stix_habil_2012}, \S 10.3). Assume that $P$ is a non trivial Severi-Brauer variety of dimension $n-1$. Let $\mathcal L$ be an invertible sheaf on $P$ whose order in $\pic P/n$ is a multiple of the period $r$. If $\mathcal L$ admits a $n$-th root on $X\to P$, then $\univ X(\kappa) = \emptyset$. In our result however, the dimension of $P$ does not play any role.
		\item Sylvain Brochard has given an independent proof of Proposition~\ref{prop:criterion-no-sections} using the torsion of the Picard scheme and an interesting duality theory for commutative group stacks (see \cite{brochard_duality_2012}). 
	\end{enumerate}
\end{remark}

From this it is easy to give examples of smooth projective geometrically connected curves of genus at least $2$ over a finitely generated field $\kappa$ such that $\univ X(\kappa) = \emptyset$, over any field with non-trivial Brauer group.

\begin{proposition}
Let $P$ be a Brauer-Severi variety over $\kappa$ with $P(\kappa) = \emptyset$, let $r$ be its exponent, and $p$ a prime factor of $r$. Then there exists a smooth geometrically connected projective curve $X$ with a morphism $f\colon X \arr P$, and an invertible sheaf $\Lambda$ on $X$ such that $\Lambda^{\otimes p} \simeq f^{*}\cO_{P}(r)$.
\end{proposition}

\begin{proof}
The result follows from the next lemma, applied to a smooth geometrically connected projective curve $Y \subseteq P$ (for example, a complete intersection in $P$), and to the restriction of $\cO_{P}(r)$ to $Y$.
\end{proof}

\begin{lemma}
Let $Y$ be a smooth geometrically connected projective curve over $\kappa$, let $L$ be an invertible sheaf on $Y$, and $n$ a positive integer. Then there exists a morphism of smooth geometrically connected projective curves $f\colon X \arr Y$ and an invertible sheaf $\Lambda$ on $X$ with $\Lambda^{\otimes n} \simeq f^{*}L$.
\end{lemma}

\begin{proof}
Suppose that $f\colon X \arr Y$ is a morphism of smooth geometrically connected projective curves over $\kappa$. If $y\in Y$ is a closed point, and we set $f^{*}y = \sum_{x \in f^{-1}(y)} e_{x} x$ (where the equality is an equality of divisors), with $e_{x} \in \NN$. We call the \emph{branch index} of $f$ at $y$ the greatest common divisor $\rmb_{y}(f)$ of the $e_{x}$. 

Now, by multiplying $L$ by the $n\th$ power of a sufficiently ample invertible sheaf on $Y$, we may assume that $L$ is very ample. Let $D$ be a smooth divisor in its linear system (which exists, since the field $\kappa$ is infinite, because it has non-trivial Brauer group). It is sufficient to show that there exists a morphism of smooth geometrically connected projective curves $f\colon X \arr Y$ with $\rmb_{y}(f) = n$ for each $y \in D$. Let $D'$ be a smooth divisor in the linear system of $L^{\otimes(n-1)}$ which is disjoint from $D$; then $D + D'$ is a smooth divisor in the linear system of $L^{\otimes n}$. We obtain $X$ as the ramified cover of $n\th$ roots of $D + D'$ in the usual fashion.
\end{proof}

\bibliography{borne_vistoli_fgs}
\end{document}